\documentclass[10 pt]{article}
\usepackage{amsmath}
\usepackage{amssymb}
\usepackage{graphicx}
\usepackage{tikz}
\usetikzlibrary{matrix}
\date{}

\newtheorem{theorem}{Theorem}[section]
\newtheorem{definition}[theorem]{Definition}
\newtheorem{lemma}[theorem]{Lemma}
\newtheorem{corollary}[theorem]{Corollary}
\newtheorem{proposition}[theorem]{Proposition}
\newtheorem{example}[theorem]{Example}

\renewcommand{\footnotemark}{\rule{0mm}{0mm}}

\newenvironment{proof}{\vskip2pt{\it Proof.}}{$\hfill \Box$\par\bigskip}

\begin{document}

\author{Fatma Kaynarca and H. Melis Tekin Akcin}

\title{\bf On symmetric property of skew polynomial rings}

\maketitle

\begin{abstract}
Symmetric rings were introduced by Lambek to extend usual commutative ideal theory in noncommutative rings. In this paper, we study symmetric rings over which Ore extensions are symmetric. A ring $R$  is  called \emph{strongly $\sigma$-symmetric} if the skew polynomial ring $R[x;\sigma]$ is symmetric.
We consider some properties and extensions of strongly $\sigma$-symmetric rings. Then we show the relationship between strongly $\sigma$-symmetric rings and other classes of rings. We next argue the polynomial extensions over strongly $\sigma$-symmetric rings. Moreover, we prove that if $R$ is a $\sigma$-rigid ring, then $R[x]/(x^n)$ is a strongly $\bar{\sigma}$-symmetric ring, where $\sigma$ is an endomorphism of $R$, $(x^n)$ is the ideal generated by $x^n$ and $n$ is a positive integer; and that if the classical left quotient ring $Q(R)$ of $R$ exists, then $R$ is $\sigma$-symmetric if and only if $Q(R)$ is strongly $\bar{\sigma}$-symmetric.

$\noindent $


\end{abstract}

\section{Introduction}

\indent\indent Throughout this paper, $R$ denotes an associative ring with identity and $\sigma$ denotes a nonzero and non-identity endomorphism, unless otherwise stated. We denote the polynomial ring with an indeterminate $x$ over $R$ by $R[x]$ and the degree of $f(x)\in R[x]$ by $deg\, f$.

Recall that a ring $R$ is called \emph{reduced} if it has no nonzero nilpotent elements. In \cite{C}, Cohn introduced the notion of a reversible ring as a generalization of commutativity. A ring $R$ is called \emph{reversible}, if whenever $a,b\in R$ satisfy $ab=0$, then $ba=0$. Anderson and Camillo \cite{AC} used the notation $ZC_{2}$ for reversible rings. While Krempa and Niewieczerzal \cite{KN} used the term $C_{0}$ for it.
This classes of rings have also found application in K\"{o}the's conjecture. Cohn proved that K\"{o}the's Conjecture is true for the class of reversible rings.

A stronger condition than 'reversibility' was defined by Lambek in \cite{L}. A ring $R$ is called {\it symmetric} if $abc=0$ implies $acb=0$ for $a, b, c\in R$. Anderson and Camillo \cite{AC} used the notation $ZC_{3}$ for symmetric rings.
Lambek also proved in \cite[Proposition 1]{L} that a ring $R$ is symmetric if and only if $r_1r_2\cdots r_n=0$ implies $r_{\sigma(1)}r_{\sigma(2)}\cdots r_{\sigma(n)}=0$ for any permutation $\sigma$ of the set $\{1,2,\cdots, n\}$ and $r_i\in R$ for all $i=1,\ldots, n$.
It is clear that commutative rings are symmetric and symmetric rings are reversible. But, in general, reversible rings need not be symmetric by \cite[Examples 5]{M} and symmetric rings need not be commutative by \cite[Example II.5]{AC}.

According to Krempa \cite{Kr}, an endomorphism $\sigma$ of a ring $R$ is called {\it rigid} if $a\sigma(a)=0$ implies $a=0$ for $a\in R$. A ring $R$ is called {\it $\sigma$-rigid} if there exists a rigid endomorphism $\sigma$ of $R$. Note that any rigid endomorphism of a ring is a monomorphism and $\sigma$-rigid rings are reduced by \cite[Propositon 5]{HKKrigid}.

Rege and Chhawchharia \cite{RC} introduced the notion of an Armendariz ring which is a generalization of a reduced ring. A ring $R$ is called {\it Armendariz} if whenever any polynomials $f(x)=a_0+a_1x+\cdots +a_mx^{m}, g(x)=b_0+b_1x+\cdots +b_nx^{n}\in R[x]$ satisfy $f(x)g(x)=0$, then $a_ib_j=0$ for each $i$ and $j$. The name `Armendariz' was given since it was Armendariz who showed that a reduced ring satisfies this condition.
Armendariz rings are also helpful to understand the relation between the annihilators of the ring and the annihilators of the polynomial ring $R[x]$.

For a ring $R$ equipped with an endomorphism $\sigma: R\rightarrow R$, a \emph{skew polynomial ring} $R[x;\sigma]$  over the coefficient ring $R$ (also called an \emph{Ore extension of endomorphism type}) is the ring obtained by giving the polynomial ring over $R$ with the new multiplication $xr=\sigma(r)x$ for all $r\in R$.
This property makes the study of Ore extensions of endomorphism type more difficult than that of polynomial rings.
Let $\sigma:R\rightarrow R$ be an endomorphism of a ring $R$. For any skew polynomial ring $R[x;\sigma]$ of $R$, we have $\sigma(1)=1$ since $1.x=x.1=\sigma(1)x$.

Armendariz property of a ring is extended to skew polynomial rings by considering the polynomials in $R[x;\sigma]$ instead of $R[x]$ (see \cite{HKR} and \cite{HKK} for more details). For an endomorphism $\sigma$ of a ring $R$, $R$ is called \emph{$\sigma$-Armendariz} (resp., \emph{$\sigma$-skew Armendariz}) if for $p(x)=\sum^{m}_{i=0}a_{i}x^{i}$ and $q(x)=\sum^{n}_{j=0}b_{j}x^{j}$ in $R[x;\sigma]$, $p(x)q(x)=0$ implies $a_{i}b_{j}=0$ (resp., $a_{i}\sigma^{i}(b_{j})=0$) for all $0\leq i\leq m$ and $0\leq j\leq n$.

Kim and Lee showed in \cite[Example 2.1]{KL} that polynomial rings over reversible rings need not be reversible. Following \cite{YL}, Yang and Liu consider reversible rings over which polynomial rings are reversible and called them \emph{strongly reversible}. According to Bell \cite{B}, a one-sided ideal $I$ of a ring $R$ is said to have the \emph{insertion-of factors-principle} (or simply \emph{IFP}) if $ab\in I$ implies $aRb\subseteq I$ for $a, b\in R$. Hence a ring $R$ is called an {\it IFP ring} if the zero ideal of $R$ has the IFP. Also note that polynomial rings over IFP rings need not be IFP by \cite[Example 2]{HLS}.
Following \cite{KLY}, Kwak et. al. called a ring $R$ \emph{strongly IFP} if $R[x]$ has IFP. Huh et. al. also proved in \cite[Example 3.1]{HKKL}  that polynomial rings are over symmetric rings need not be symmetric. In \cite{EA}, Eltiyeb and Ayoub investigated symmetric rings over which polynomial rings are symmetric.

Another approach to generalize reversible and IFP  properties is obtained by considering the properties on Ore extensions of endomorphism type.
Following \cite{JKKL}, Jin et. al. called a ring $R$ {\it strongly $\sigma$-skew reversible} if the skew polynomial ring $R[x;\sigma]$ is reversible and in \cite{BHKKL}, Ba\c{s}er et. al. called a ring $R$ {\it strongly $\sigma$-IFP} if the skew polynomial ring $R[x;\sigma]$ has $IFP$.

Motivated by the above, in this paper, we introduce a new class of rings which is called {\it strongly $\sigma$-symmetric} ring to extend the symmetry property on skew polynomials. The following diagram describes all known implications. Also note that no other implications in the diagram hold in general.

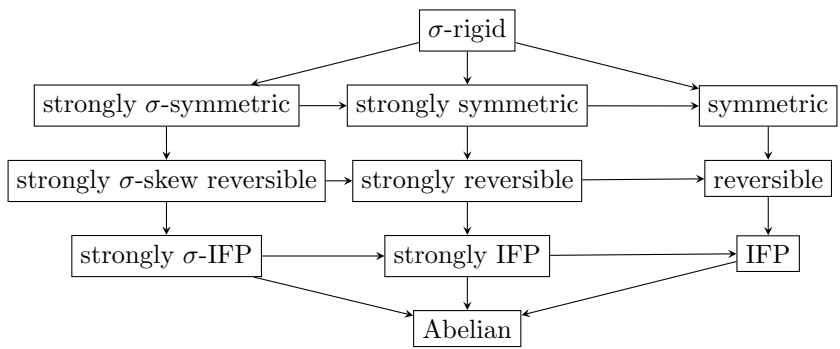
\begin{center}
\begin{tikzpicture}[->,>=stealth]
    \matrix (dag)  [matrix of nodes,%
    nodes={rectangle,draw},
    column sep={4cm,between origins},
    row sep={1cm,between origins},
    ampersand replacement=\&] { \& |(1)| $\sigma$-rigid  \&   \\
                                   |(11)| strongly $\sigma$-symmetric   \& |(12)| strongly symmetric     \&   |(13)| symmetric \\
                                   |(21)| strongly $\sigma$-skew reversible  \& |(22)| strongly reversible \&  |(23)| reversible \\
                                   |(31)| strongly $\sigma$-IFP  \& |(32)| strongly IFP  \&  |(33)| IFP \\
                                   \&    |(4)| Abelian \&  \\  };
    \draw (1) to (11);
    \draw (1) to (12);
    \draw (1) to (13);
    \draw (11) to (12);
    \draw (12) to (13);
    \draw (11) to (21);
    \draw (12) to (22);
    \draw (13) to (23);
    \draw (21) to (22);
    \draw (21) to (31);
    \draw (22) to (32);
    \draw (31) to (32);
    \draw (22) to (23);
    \draw (32) to (33);
    \draw (23) to (33);
    \draw (31) to (4);
    \draw (32) to (4);
    \draw (33) to (4);
  \end{tikzpicture}
\end{center}

In section 2, we examine the relationships between several classes of rings and strongly $\sigma$-symmetric rings and prove some statements about the links given in the above diagram. We also provide some examples of strongly $\sigma$-symmetric rings and counterexamples to several naturally raised situations.

In Section 3, as suggested by the literature, there is considerable interest whether strongly $\sigma$-symmetric property is preserved under extensions. We first examine when polynomial rings over strongly $\sigma$-symmetric rings are again strongly $\bar{\sigma}$-symmetric. Next, we prove that if $R$ is a $\sigma$-rigid ring, then $R[x]/(x^n)$ is a strongly $\bar{\sigma}$-symmetric ring, where $\sigma$ is an endomorphism of $R$, $(x^n)$ is the ideal generated by $x^n$ and $n$ is a positive integer.

In Theorem \ref{Quotientring}, we prove that if the classical left quotient ring $Q(R)$ of $R$ exists, then $R$ is $\sigma$-symmetric if and only if $Q(R)$ is strongly $\bar{\sigma}$-symmetric. In Proposition \ref{delta}, we obtain the results proved in \cite[Proposition 3.6]{BHKKL} and \cite[Proposition 3.8]{JKKL}  without the condition '$\sigma(u)=u$' for any central regular element $u$. Hence, we get a direct generalization of \cite[Lemma 3.2]{HKKL} without any restriction on the endomorphism $\sigma$. Moreover, several known results relating to symmetric rings can be obtained as corollaries of our results.

\section{Strongly $\sigma$-symmetric rings and  related properties }

\indent \indent
The present work is devoted to study ring-theoretical properties of strongly $\sigma$-symmetric rings.
Our focus in this section is to introduce the concept of a strongly $\sigma$-symmetric ring for an endomorphism $\sigma$ and investigate its properties. Firstly, we begin with the following example which illustrates the need to introduce the symmetry property of skew polynomial rings.

\begin{example}\label {sss}
{\rm   Consider the ring $R=\mathbb{Z}_{2}\oplus \mathbb{Z}_{2}$ with the usual addition and multiplication. Then we know that $R$ is symmetric since $R$ is reduced. Let $\sigma:R\rightarrow R$ be an endomorphism of $R$ defined by $\sigma((a,b))=(b,a)$. For $p(x)=(1,1)$, $q(x)=(1,0)$ and  $r(x)=(0,1)x$ in $R[x;\sigma]$ we have $p(x)q(x)r(x)=0$, but $p(x)r(x)q(x)=(1,1)(0,1)x(1,0)=(1,1)(0,1)\sigma((1,0))x=(0,1)x\neq (0,0)$. Thus $R[x;\sigma]$ is not symmetric.}
\end{example}
\indent Motivated by this example, we can give the following definition.

\begin{definition} \label{def}
{\rm  Let $R$ be a ring and $\sigma$ be an endomorphism of $R$. Then $R$ is called  {\it strongly $\sigma$-symmetric} if $R[x;\sigma]$ is symmetric.}
\end{definition}

Every  strongly $\sigma$-symmetric ring is symmetric, but the converse is not true by Example~\ref{sss}.  Any $\sigma$-rigid ring (i.e., $R[x;\sigma]$ is reduced) is clearly strongly $\sigma$-symmetric. However, there exists a  strongly $\sigma$-symmetric ring which is not $\sigma$-rigid by \cite[Example 3.4]{GM}. It is clear that any domain $R$ with a monomorphism $\sigma$ is strongly $\sigma$-symmetric since $R$ is $\sigma$-rigid. Note that every subring $S$ of a strongly $\sigma$-symmetric ring with $\sigma(S)\subseteq S$ is also strongly $\sigma$-symmetric. Any strongly $\sigma$-symmetric ring is clearly strongly $\sigma$-IFP, but the converse is not true in general, by the following example.

\begin{example}\label{IFPnotsym}
\rm Consider the polynomial ring $R=\mathbb{Z}[x]$. Let $\sigma: R\rightarrow R$ be an endomorphism of $R$ defined by $\sigma(f(x))=f(0)$ for $f(x)\in \mathbb{Z}[x]$.
Then by \cite[Example 2.3]{BHKKL}, we know that $R$ is strongly $\sigma$-IFP. On the other hand, for the polynomials $p(y)=1+x$, $q(y)=xy$ and $r(y)=x \in \mathbb{Z}[x][y;\sigma]$ we have $p(y)q(y)r(y)=0$, but $p(y)r(y)q(y)\neq 0$. Therefore, $R$ is not strongly $\sigma$-symmetric.
\end{example}

Following Kwak \cite[Definition 2.1]{K}, an endomorphism $\sigma$ of a ring
$R$ is called {\it right (left) symmetric} if whenever $abc=0$ for $a, b, c\in R$, then
$ac\sigma(b)=0$ ($\sigma(b)ac=0$). A ring $R$ is called {\it right (left) $\sigma$-symmetric}
if there exists a right (left) endomorphism $\sigma$ of $R$. $R$ is called \emph{$\sigma$-symmetric} if it is both right and left $\sigma$-symmetric.

\begin{proposition} \label{sym}
Let $R$ be a ring and $\sigma$ be an endomorphism of $R$. If $R$ is strongly $\sigma$-symmetric, then $R$ is both symmetric and $\sigma$-symmetric.
\end{proposition}
\begin{proof}
Let $R$ be a strongly $\sigma$-symmetric ring, then it is clear that $R$ is symmetric. Let $abc=0$ for $a, b, c\in R$ and consider the polynomials $p(x)=a$, $q(x)=b$ and $r(x)=cx$ in $R[x;\sigma]$. Then $p(x)q(x)r(x)=0$.
Since $R$ is strongly $\sigma$-symmetric, we have $p(x)r(x)q(x)=0$ and so $ac\sigma(b)=0$. Then we obtain that $R$ is right $\sigma$-symmetric. Moreover, we have $\sigma(b)ac=0$ since $R$ is reversible. Therefore, $R$ is  left $\sigma$-symmetric.
\end{proof}

Let $\sigma$ be an endomorphism of a ring $R$. Then $R$ is called \emph{$\sigma$-skew quasi Armendariz} \cite[Definition 2.1]{HKL1} if whenever $p(x)R[x;\sigma]q(x)=0$ for $p(x)=\sum^{m}_{i=0}a_ix^{i}$ and $q(x)=\sum^{n}_{j=0}b_{j}x^{j}\in R[x;\sigma]$, then $a_{i}R\sigma^{i} (b_{j})=0$ for all $0\leq i\leq m$ and $0\leq j\leq n$. Also note that any $\sigma$-skew Armendariz ring is clearly $\sigma$-skew quasi Armendariz, when $\sigma$ is an epimorphism by \cite{HKL1}.
But, $\sigma$-skew quasi Armendariz rings need not be $\sigma$-skew Armendariz even if $\sigma$ is an automorphism by \cite[Example 2.2(1)]{HKL1}. In the following theorem, we show that over strongly $\sigma$-symmetric rings these concepts are equivalent.

\begin{theorem}
Let $R$ be a ring and $\sigma$ be an epimorphism of $R$. If $R$ is strongly $\sigma$-symmetric, then $R$ is $\sigma$-skew Armendariz if and only if $R$ is $\sigma$-skew quasi Armendariz.
\end{theorem}
\begin{proof}
It is enough to prove that $R$ is $\sigma$-skew Armendariz when $R$ is $\sigma$-skew quasi Armendariz. Let $p(x)q(x)=0$, where $p(x)=\sum^{m}_{i=0}a_{i}x^{i}$ and $q(x)=\sum^{n}_{j=0}b_{j}x^{j}$ in $R[x;\sigma]$. It is clear that $p(x)q(x)r(x)=0$ for all $r(x)\in R[x;\sigma]$. Since $R$ is strongly $\sigma$-symmetric, we have $p(x)r(x)q(x)=0$ for all $r(x)\in R[x;\sigma]$. Then $p(x)R[x;\sigma]q(x)=~0$ and so $a_{i}R\sigma^{i}(b_{j})=0$ for all $0\leq i\leq m$ and $0\leq j\leq n$ since $R$ is $\sigma$-skew quasi Armendariz. Therefore, $a_{i}\sigma^{i}(b_{j})=0$ for all $0\leq i\leq m$ and $0\leq j\leq n$, as required.
\end{proof}

Note that $\sigma$-skew Armendariz rings and strongly $\sigma$-symmetric rings are independent of each other by the following examples.

\begin{example}\label{skewarm}
\rm (1) We consider the ring $R=\mathbb{Z}_{2}[x]$ and the endomorphism $\sigma:R\rightarrow R$ is defined by $\sigma(f(x))=f(0)$ for $f(x)\in \mathbb{Z}_{2}[x]$. Then $R$ is $\sigma$-skew Armendariz by \cite[Example 5]{HKK}. But $R$ is not strongly $\sigma$-symmetric. Indeed, for the polynomials $p(y)=1, q(y)=(\bar{1}+x)y, r(y)=x$ in $\mathbb{Z}_2[x][y;\sigma]$, we have $p(y)q(y)r(y)=0$. But $p(y)r(y)q(y)=x(\bar{1}+x)y\neq 0$ and hence, $R$ is not strongly $\sigma$-symmetric.\\
\rm (2) We consider the ring $S$ given by $S=\left\{\left( \begin{array}{cc} \bar{a} & \bar{b} \\ \bar0 & \bar{a} \\ \end{array} \right)\mid~\bar{a}, \bar{b}\in \mathbb{Z}_{4}  \right\}$. By \cite[Example 14]{HKK}, $S$ is not $I_{S}$-skew Armendariz, where $I_{S}$ is the identity map of $S$. On the other hand, it can be seen that $S$ is strongly $I_{S}$-symmetric.

\end{example}

\begin{theorem}\label{sigmaA}
Let $R$ be a $\sigma$-Armendariz ring. Then the followings are equivalent:
\begin{enumerate}
  \item [(1)] $R$ is right $\sigma$-symmetric.
  \item [(2)] $R$ is symmetric.
  \item [(3)] $R$ is strongly $\sigma$-symmetric.
\end{enumerate}
\end{theorem}
\begin{proof}
$(1)\Rightarrow (2)$ Suppose that $R$ is right $\sigma$-symmetric. Let $a, b, c\in R$ such that $abc=0$. Since $R$ is right $\sigma$-symmetric, we have $ac\sigma(b)=0$. Then $acb=0$ by \cite[Proposition 1.3(ii)]{HKR}. Hence, $R$ is symmetric.\\
$(2)\Rightarrow (3)$ It is clear by \cite[Theorem 3.6(i)]{HKR}.\\
$(3)\Rightarrow (1)$ Suppose that $R$ is strongly $\sigma$-symmetric and $abc=0$ for $a, b, c\in R$. Let $p(x)=a, q(x)=b$ and $r(x)=cx$ in $R[x;\sigma]$. Then $p(x)q(x)r(x)=0$ and hence, we get $p(x)r(x)q(x)=0$ since $R$ is strongly $\sigma$-symmetric. Therefore, $ac\sigma(b)=0$ and we obtain that $R$ is right $\sigma$-symmetric.
\end{proof}

Notice that the condition ``$R$ is $\sigma$-Armendariz'' in Theorem~\ref{sigmaA} is not superfluous by Example~\ref{skewarm}(1). In \cite[Example 1.9]{HKR}, it is proved that $R=\mathbb{Z}_{2}[x]$ is not $\sigma$-Armendariz. Also note that $R$ is a commutative domain and hence, we get $R$ is symmetric and right $\sigma$-symmetric for any endomorphism.

\begin{corollary}(\cite[Proposition 3.4]{HKKL})
Let $R$ be an Armendariz ring, then $R$ is symmetric if and only if $R[x]$ is symmetric.
\end{corollary}

The following lemma is clear by \cite[Lemma 2.3]{JKKL}, since any strongly $\sigma$-symmetric ring is strongly $\sigma$-skew reversible. For the sake of completeness, we include the statements.

\begin{lemma}\label{Lemma1-23.10.18}
Let $R$ be a strongly $\sigma$-symmetric ring. Then we have the following results:
\begin{enumerate}
\item $R$ is symmetric.
\item $\sigma$ is a monomorphism.
\item For any $a,b\in R$ and nonnegative integer $m$ and $n$, we have $a\sigma^{m}(b)=0\Leftrightarrow ab=0 \Leftrightarrow ba=0 \Leftrightarrow \sigma^{m}(b)\sigma^{n}(a)=0\Leftrightarrow \sigma^{n}(a)\sigma^{m}(b)=0$.
\item $R$ is abelian and $\sigma(e)=e$ for any $e^2=e\in R$.
\end{enumerate}
\end{lemma}

Following \cite{HM}, a ring $R$ with an endomorphism $\sigma$ is called \emph{$\sigma$-compatible} if for each $a, b\in R$, $ab=0\Leftrightarrow a\sigma(b)=0$. It is a well-known fact that if $R$ is a $\sigma$-compatible ring, then $\sigma$ is a monomorphism.

\begin{lemma}\label{Lemma2-23.10.18}
Let $R$ be a ring and $\sigma$ be an endomorphism of $R$. If $R$ is strongly $\sigma$-symmetric, then $R$ is $\sigma$-compatible.
\end{lemma}
\begin{proof}
If $R$ is strongly $\sigma$-symmetric, then $R$ is strongly $\sigma$-skew reversible and by \cite[Corollary 2.4(1)]{JKKL}, we get $R$ is $\sigma$-compatible.
\end{proof}

But $\sigma$-compatible rings need not be strongly $\sigma$-symmetric by \cite[Example 2.11]{JKKL}. In the following theorem, we show the relation between $\sigma$-compatible rings and strongly $\sigma$-symmetric rings.

\begin{theorem}\label{THM1 23.10.18}
Let $R$ be a ring and $\sigma$ be an endomorphism of $R$. Assume that $R$ is $\sigma$-skew Armendariz. Then $R$ is symmetric and $\sigma$-compatible if and only if $R$ is strongly $\sigma$-symmetric.
\end{theorem}
\begin{proof}
It suffices to show the necessity by using Lemma \ref{Lemma2-23.10.18} and the fact that strongly $\sigma$-symmetric property is inherited by its subrings.
Suppose that $R$ is symmetric and $\sigma$-compatible. Let $p(x)q(x)r(x)=0$, where $p(x)=\sum^{m}_{i=0}a_{i}x^{i}$, $q(x)=\sum^{n}_{j=0}b_{j}x^{j}$ and $r(x)=\sum^{l}_{k=0}c_{k}x^{k}$ in $R[x;\sigma]$. Then we have $a_{i}\sigma^{i}(b_{j})\sigma^{i+j}(c_{k})=0$ for all $0\leq i\leq m, 0\leq j\leq n$ and $0\leq k\leq l$. Since $R$ is $\sigma$-compatible, we obtain $a_{i}b_{j}c_{k}=0$. Therefore, $a_{i}c_{k}b_{j}=0$ since $R$ is symmetric. Hence, $a_{i}\sigma^{i}(c_{k})\sigma^{i+j}(b_{j})=0$ for all $0\leq i\leq m, 0\leq j\leq n$ and $0\leq k\leq l$ and this implies that $p(x)r(x)q(x)=0$, as required.
\end{proof}

The conditions "$R$ is a $\sigma$-compatible ring" and "$R$ is $\sigma$-skew Armendariz ring" in Theorem~\ref{THM1 23.10.18} can not be dropped by the following examples.

\begin{example}\rm
Consider the polynomial ring $R=\mathbb{Z}_{2}[x]$ and the endomorphism  $\sigma:R\rightarrow R$ defined by $\sigma(f(x))=f(0)$ for $f(x)\in \mathbb{Z}_{2}[x]$. By \cite[Example 5]{HKK}, $R$ is $\sigma$-skew Armendariz and $R$ is symmetric since $R$ is a domain. On the other hand, $R$ is not strongly $\sigma$-symmetric by Example~\ref{skewarm}(1) and $R$ is not $\sigma$-compatible since $\sigma$ is not a monomorphism.
\end{example}

\begin{example} \label{noncomalg} \rm
We use the ring constructions in \cite[Example 3.1]{HKKL} and \cite[Example 2.5(1)]{JKKL}. Let $\mathbb{Z}_{2}$ be the ring of integers modulo 2 and $A=\mathbb{Z}_{2}[a_0, a_1, a_2, b_0, b_1, b_2,c]$ be the  the free algebra of polynomials with noncommuting indeterminates $a_0,a_1, a_2, b_0, b_1, b_2, c$ over $\mathbb{Z}_{2}$. Define an automorphism $\sigma$ of $A$ by
$$a_0,a_1, a_2, b_0, b_1, b_2, c\mapsto b_0,b_1, b_2, a_0, a_1, a_2, c$$
respectively. Let $B$ be the set of all polynomials with zero constant terms in $A$ and $I$ be the ideal of $A$ generated by
$$a_0b_0,~~b_0a_0,~~a_2b_2,~~b_2a_2,~~a_0a_0,~~a_2a_2,~~b_0b_0,~~b_2b_2$$
$$a_0rb_0,~~b_0ra_0,~~a_2rb_2,~~b_2ra_2,~~a_0ra_0,~~a_2ra_2,~~b_0rb_0,~~b_2rb_2,~~r_1r_2r_3r_4$$
$$ a_0b_1+a_1b_0,~~b_0a_1+b_1a_0,~~a_1b_2+a_2b_1,~~b_1a_2+b_2a_1,~~a_0a_1+a_1a_0,~~b_0b_1+b_1b_0,~~a_1a_2+a_2a_1,~~b_1b_2+b_2b_1$$
$$a_0b_2+a_1b_1+a_2b_0,~~b_0a_2+b_1a_1+b_2a_0,~~a_0a_2+a_1a_1+a_2a_0,~~b_0b_2+b_1b_1+b_2b_0$$
$$(a_0+a_1+a_2)r(b_0+b_1+b_2),~~(b_0+b_1+b_2)r(a_0+a_1+a_2),~~(a_0+a_1+a_2)r(a_0+a_1+a_2),~~(b_0+b_1+b_2)r(b_0+b_1+b_2)$$
for $r, r_1, r_2, r_3, r_4\in B$. Then clearly $B^{4}\subseteq I$. Set $R=A/I$. Since $\sigma(I)\subseteq I$, we can induce an automorphism $\bar\sigma$ of $R$ defined by $\bar\sigma(s+I)=\sigma(s)+I$ for $s\in A$. Note that $\sigma^{2}=1_R$, where $1_R$ denotes the identity endomorphism of $R$. By \cite[Example 2.5]{JKKL}, $R$ is $\bar\sigma$-compatible, but not strongly $\bar\sigma$-skew reversible. Therefore, $R$ is not strongly $\bar\sigma$-symmetric. Moreover, by \cite[Example 2.10(1)]{HKKL}, $R$ is not $\bar\sigma$-skew Armendariz.

\noindent Now we show that $R$ is symmetric. We shall call $f\in A$ a \emph{monomial} of degree $n$ if it is a product of exactly $n$ number of indeterminates, and let $H_{n}$ be the set of all linear combinations of monomials of degree $n$ over $\mathbb{Z}_{2}$. Then $H_n$ is finite for any $n$. Moreover, the ideal $I$ is homogeneous (i.e., if $\sum^t_{i=1}r_{i}\in I$ with $r_i\in H_i$ then every $r_i$ is in I) by the construction.\\
\textbf{Claim:} If $f_1g_1h_1\in I$ for $f_1, g_1, h_1\in H_1$ then $f_1h_1g_1\in I$.\\
\textbf{Proof of the claim:}
Based on the construction of $I$ we have the following cases when $f_1g_1h_1\in I$ for $f_1, g_1, h_1\in H_1$; \\
\emph{Case 1:}
$$(f_1=a_0, g_1=b_0, h_1=r),~~(f_1=a_0, g_1=r, h_1=b_0),~~(f_1=b_0, g_1=a_0, h_1=r),$$
$$(f_1=b_0, g_1=r, h_1=a_0),~~(f_1=a_0, g_1=a_0, h_1=r),~~(f_1=a_0, g_1=r, h_1=a_0),$$
$$(f_1=b_0, g_1=b_0, h_1=r),~~(f_1=b_0, g_1=r, h_1=b_0),~~(f_1=r, g_1=a_0, h_1=b_0),$$
$$(f_1=r, g_1=b_0, h_1=a_0),~~(f_1=a_2, g_1=b_2, h_1=r),~~(f_1=a_2, g_1=r, h_1=b_2),$$
$$(f_1=b_2, g_1=a_2, h_1=r),~~(f_1=b_2, g_1=r, h_1=a_2),~~(f_1=a_2, g_1=a_2, h_1=r),$$
$$(f_1=a_2, g_1=r, h_1=a_2),~~(f_1=b_2, g_1=b_2, h_1=r),~~(f_1=b_2, g_1=r, h_1=b_2),$$
$$(f_1=r, g_1=a_2, h_1=b_2),~~(f_1=r, g_1=b_2, h_1=a_2),~~(f_1=r, g_1=a_0, h_1=a_0),$$
$$(f_1=r, g_1=a_2, h_1=a_2),~~(f_1=r, g_1=b_0, h_1=b_0),~~(f_1=r, g_1=b_2, h_1=b_2),$$
$$(f_1=a_0+a_1+a_2, g_1=r, h_1=b_0+b_1+b_2),~~(f_1=a_0+a_1+a_2, g_1=b_0+b_1+b_2, h_1=r),$$
$$(f_1=r, g_1=a_0+a_1+a_2, h_1=b_0+b_1+b_2),~~(f_1=b_0+b_1+b_2, g_1=r, h_1=a_0+a_1+a_2),$$
$$(f_1=b_0+b_1+b_2, g_1=a_0+a_1+a_2, h_1=r),~~(f_1=r, g_1=b_0+b_1+b_2, h_1=a_0+a_1+a_2),$$
$$(f_1=a_0+a_1+a_2, g_1=a_0+a_1+a_2, h_1=r),~~(f_1=a_0+a_1+a_2, g_1=r, h_1=a_0+a_1+a_2),$$
$$(f_1=r, g_1=a_0+a_1+a_2, h_1=a_0+a_1+a_2),~~(f_1=b_0+b_1+b_2, g_1=b_0+b_1+b_2, h_1=r),$$
$$(f_1=b_0+b_1+b_2, g_1=r, h_1=b_0+b_1+b_2),~~(f_1=r, g_1=b_0+b_1+b_2, h_1=b_0+b_1+b_2)$$
for $r\in H_1$ and these cases are clear by the construction of $I$.\\
\emph{Case 2:}
\begin{description}
  \item[] If ($f_1=a_0, g_1=b_1, h_1=a_0$),~then $f_1g_1h_1\in I$ and $f_1h_1g_1\in I$, since $a_0b_1a_0=a_0b_1a_0+a_1b_0a_0=\\(a_0b_1+a_1b_0)a_0\in I$ and $a_0a_0b_1\in I$.
  \item[] If ($f_1=b_0, g_1=a_1, h_1=b_0$),~then $f_1g_1h_1\in I$ and $f_1h_1g_1\in I$, since $b_0a_1b_0=b_0a_1b_0+b_1a_0b_0=\\(b_0a_1+b_1a_0)b_0\in I$ and $b_0b_0a_1\in I$.
  \item[] If ($f_1=b_2, g_1=a_1, h_1=b_2$),~then $f_1g_1h_1\in I$ and $f_1h_1g_1\in I$, since $b_2a_1b_2=b_2a_1b_2+b_2a_2b_1=\\b_2(a_1b_2+a_2b_1)\in I$ and $b_2b_2a_1\in I$.
  \item[] If ($f_1=a_2, g_1=b_1, h_1=a_2$),~then $f_1g_1h_1\in I$ and $f_1h_1g_1\in I$, since $a_2b_1a_2=a_2b_1a_2+a_2b_2a_1=\\a_2(b_1a_2+b_2a_1)\in I$ and $a_2a_2b_1\in I$.
  \item[] If ($f_1=b_0, g_1=a_1, h_1=a_0$),~then $f_1g_1h_1\in I$ and $f_1h_1g_1\in I$, since $b_0a_1a_0=b_0a_1a_0+b_0a_0a_1=\\b_0(a_0a_1+a_1a_0)\in I$ and $b_0a_0a_1\in I$.
  \item[] If ($f_1=a_0, g_1=a_1, h_1=b_0$),~then $f_1g_1h_1\in I$ and $f_1h_1g_1\in I$, since $a_0a_1b_0=a_0a_1b_0+a_1a_0b_0=\\(a_0a_1+a_1a_0)b_0\in I$ and $a_0b_0a_1\in I$.
  \item[] If ($f_1=a_0, g_1=b_1, h_1=b_0$),~then $f_1g_1h_1\in I$ and $f_1h_1g_1\in I$, since $a_0b_1b_0=a_0b_1b_0+a_0b_0b_1=\\a_0(b_0b_1+b_1b_0)\in I$ and $a_0b_0b_1\in I$.
  \item[] If ($f_1=b_0, g_1=b_1, h_1=a_0$),~then $f_1g_1h_1\in I$ and $f_1h_1g_1\in I$, since $b_0b_1a_0=b_0b_1a_0+b_1b_0a_0=\\(b_0b_1+b_1b_0)a_0\in I$ and $b_0a_0b_1\in I$.
  \item[] If ($f_1=a_2, g_1=a_1, h_1=b_2$),~then $f_1g_1h_1\in I$ and $f_1h_1g_1\in I$, since $a_2a_1b_2=a_2a_1b_2+a_1a_2b_2=\\(a_1a_2+a_2a_1)b_2\in I$ and $a_2b_2a_1\in I$.
  \item[] If ($f_1=b_2, g_1=a_1, h_1=a_2$),~then $f_1g_1h_1\in I$ and $f_1h_1g_1\in I$, since $b_2a_1a_2=b_2a_1a_2+b_2a_2a_1=\\b_2(a_1a_2+a_2a_1)\in I$ and $b_2a_2a_1\in I$.
  \item[] If ($f_1=a_2, g_1=b_1, h_1=b_2$),~then $f_1g_1h_1\in I$ and $f_1h_1g_1\in I$, since $a_2b_1b_2=a_2b_1b_2+a_2b_2b_1=\\a_{2}(b_1b_2+b_2b_1)\in I$ and $a_2b_2b_1\in I$.
  \item[] If ($f_1=b_2, g_1=b_1, h_1=a_2$),~then $f_1g_1h_1\in I$ and $f_1h_1g_1\in I$, since $b_2b_1a_2=b_2b_1a_2+b_1b_2a_2=\\(b_1b_2+b_2b_1)a_2\in I$ and $b_2a_2b_1\in I$.
\end{description}
Thus we obtain that $f_1g_1h_1\in I$ implies $f_1h_1g_1\in I$ for any case, and it proves the claim. Now let $f=f_1+f_2+f_3+f_4,~g=g_1+g_2+g_3+g_4,~h=h_1+h_2+h_3+h_4$ with $f_i, g_i, h_i\in H_i$ for $i=1, 2, 3$ and $f_4, g_4, h_4\in I$. Let $fgh\in I$ with $f, g, h\in A$. We want to see that $R$ is symmetric. Since each monomial of degree $\geq 4$ is contained in $I$, then we have $fgh=f_1g_1h_1+h'\in I$, where $h'\in I$. Hence, $f_1g_1h_1\in I$ and $f_1h_1g_1\in I$ by the claim. Therefore, we have $fhg\in I$ as required.
\end{example}

Recall that for a ring $R$ and an endomorphism $\sigma$ of $R$, an ideal $I$ of $R$ is called a \emph{$\sigma$-ideal} if $\sigma(I)\subseteq I$.

\begin{definition}\cite[Definition 1.1]{KKH}
\rm Let $\sigma$ be an automorphism of a ring $R$. For a $\sigma$-ideal $I$ of $R$, $I$ is called \emph{strongly $\sigma$-semiprime ideal} of $R$ if $aR\sigma(a)\subseteq I$ implies $a\in I$ for any $a\in R$. $R$ is called a \emph{strongly $\sigma$-semiprime} ring if the zero ideal is strongly $\sigma$-semiprime.
\end{definition}

For an automorphism $\sigma$ of $R$, every $\sigma$-rigid ring is strongly $\sigma$-semiprime. Also recall that any $\sigma$-rigid ring is strongly $\sigma$-symmetric. The following proposition shows when strongly $\sigma$-symmetric rings are $\sigma$-rigid.

\begin{proposition}\label{sssp}
Let $R$ be a ring and $\sigma$ be an automorphism of $R$. Then $R$ is $\sigma$-rigid if and only if $R$ is strongly $\sigma$-semiprime and $R$ is strongly $\sigma$-symmetric.
\end{proposition}
\begin{proof} Suppose that $R$ is strongly $\sigma$-semiprime and strongly $\sigma$-symmetric. Let $a\sigma(a)=0$ for $a\in R$. Now, consider the polynomials $p(x)=ax$ and $q(x)=a$ in $R[x;\sigma]$. Then we have $p(x)q(x)r(x)=0$ for any $r(x)\in R[x;\sigma]$. Since $R$ is strongly $\sigma$-symmetric, we have $p(x)r(x)q(x)=0$ for any $r(x)\in R[x;\sigma]$. This implies that $aR\sigma(a)=0$ since $\sigma$ is onto. Therefore, $a=0$ since $R$ is strongly $\sigma$-semiprime.
\end{proof}

The concepts of strongly $\sigma$-symmetric rings and strongly $\sigma$-semiprime rings do not imply each other by the following examples.

\begin{example}\rm
(1) Let $F$ be a field with char$(F)\neq 2$. Consider the $2\times 2$ matrix ring over $F$, $R=Mat_{2}(F)$, and let $\sigma:R\rightarrow R$ be an endomorphism defined by $\sigma\left( \left( \begin{array}{cc} a & b \\ c & d \\ \end{array} \right)  \right) = \left( \begin{array}{cc} a & -b \\ -c & d \\ \end{array} \right)$. Then it is proved in \cite[Example 2.1]{KKH} that $R$ is strongly $\sigma$-semiprime. Let $p(x)=\left( \begin{array}{cc} 0 & 1 \\  0 & 1 \\  \end{array}  \right)$, $q(x)=\left( \begin{array}{cc} 1 & 1 \\  1 & 1 \\  \end{array}  \right)x$ and $r(x)=\left( \begin{array}{cc} 1 & 1 \\  1 & 1 \\  \end{array}  \right)\in R[x;\sigma]$. Then we have $p(x)q(x)r(x)=0$, but
$$p(x)r(x)q(x)=\left( \begin{array}{cc} 0 & 1 \\  0 & 1 \\  \end{array}  \right)\left( \begin{array}{cc} 1 & 1 \\  1 & 1 \\  \end{array}  \right)\left( \begin{array}{cc} 1 & 1 \\ 1 & 1 \\  \end{array}  \right)x=\left( \begin{array}{cc} 2 & 2 \\  2 & 2 \\  \end{array}  \right)x\neq O.$$
Thus $R$ is not strongly $\sigma$-symmetric. Moreover, $R$ is not $\sigma$-rigid since $\begin{pmatrix}
1 & 1\\
1 & 1
\end{pmatrix}\sigma\left(\begin{pmatrix}
1 & 1\\
1 & 1
\end{pmatrix}\right)=0$ and $\begin{pmatrix}
1 & 1\\
1 & 1
\end{pmatrix}\neq 0$.
This example also shows that being strongly $\sigma$-symmetric is not superfluous in Theorem \ref{sssp}.\\
(2) Let us consider the ring $R=\left\{ \left( \begin{array}{cc} a & t \\ 0 & a \\ \end{array} \right)\mid a\in \mathbb{Z},\ t\in \mathbb{Q} \right\}$ where $\mathbb{Z}$ and $\mathbb{Q}$ are the set of all integers and rational numbers, respectively. Let $\sigma$ be the identity endomorphism of $R$. $R$ is not strongly $\sigma$-semiprime since
$$\left( \begin{array}{cc} 0 & 1 \\ 0 & 0 \\ \end{array} \right)
 \left( \begin{array}{cc} a & t \\ 0 & a \\ \end{array} \right)
 \sigma\left(\left( \begin{array}{cc} 0 & 1 \\ 0 & 0 \\ \end{array} \right)\right)=0$$
for any $a\in \mathbb{Z}$ and $t\in \mathbb{Q}$. On the other hand, suppose that $p(x)q(x)r(x)=0$ for the polynomials
$$p(x)=\sum^{m}_{i=0}\left( \begin{array}{cc} a_i & t_i \\ 0 & a_i \\ \end{array} \right)x^{i},\
  q(x)=\sum^{n}_{j=0}\left( \begin{array}{cc} b_j & t'_j\\ 0 & b_j \\ \end{array} \right)x^{j} \text{ and }
  r(x)=\sum^{l}_{k=0}\left( \begin{array}{cc}c_k & t''_k\\ 0 & c_k \\ \end{array} \right)x^{k}$$
in $R[x;\sigma]$. We have
\begin{align}
0=&p(x)q(x)r(x)\notag \\
 =&\sum^{m+n+l}_{t=0}\bigg( \sum_{t=s+k} \bigg( \sum_{s=i+j}
 \left( \begin{array}{cc} a_i & t_i \\ 0 & a_i \\ \end{array} \right)
 {\bar\sigma}^{i}\left( \begin{array}{cc} b_j & t'_j\\ 0 & b_j \\ \end{array} \right)
 {\bar\sigma}^{s}\left( \begin{array}{cc} c_k & t''_k\\ 0 & c_k \\ \end{array} \right)
 \bigg) \bigg)x^{t} \notag\\
 =&\sum^{m+n+l}_{t=0} \sum_{t=s+k}  \sum_{s=i+j}
 \left( \begin{array}{cc} a_ib_jc_k & a_ib_jt''_k+a_it'_jc_k+t_ib_jc_k \\ 0 & a_ib_jc_k \\ \end{array} \right)x^{t} \notag\\
 =&p(x)r(x)q(x). \notag
\end{align}
Therefore, $R$ is strongly $\sigma$-symmetric.
\end{example}

\begin{corollary} Let $R$ be a ring and $\sigma$ be an automorphism of $R$. Then the following statements are equivalent:
\begin{itemize}
\item[(1)]$R$ is strongly $\sigma$-semiprime and and strongly $\sigma$-symmetric.
\item[(2)]$R$ is $\sigma$-rigid.
\item[(3)] $R[x, x^{-1}; \sigma]$ is reduced.
\end{itemize}
\end{corollary}
\begin{proof}
$(1)\Leftrightarrow(2)$ is obtained by Theorem \ref{sssp} and $(2)\Leftrightarrow(3)$ follows from \cite[Theorem 3]{NM}.
\end{proof}

By Proposition~\ref{sssp}, we obtain a generalization of the following result.

\begin{corollary}\cite[Proposition 2.7(1)]{HKKL}
$R$ is reduced if and only if $R$ is semiprime and $R$ is symmetric.
\end{corollary}

\begin{proposition}\label{subdirectsum}
Let $R$ be a ring, $\sigma$ be an endomorphism of $R$ and $I_{\lambda}$ be an ideal of $R$ with $\sigma(I_{\lambda})\subseteq I_{\lambda}$ for all $\lambda\in \Lambda$. Let $\sigma_{\lambda}:R/I_{\lambda}\rightarrow R/I_{\lambda}$ be the induced endomorphism of $R/I_{\lambda}$. If $R$ is a subdirect sum of strongly $\sigma_{\lambda}$-symmetric rings for all $\lambda\in \Lambda$, then $R$ is a strongly $\sigma$-symmetric ring.
\end{proposition}
\begin{proof} Since $R$ is a subdirect sum of strongly $\sigma_{\lambda}$-symmetric rings, by \cite[Theorem 3]{MC2}, we have $R/I_{\lambda}$ is a strongly $\sigma_{\lambda}$-symmetric ring for all $\lambda\in \Lambda$ and $\cap_{\lambda\in \Lambda}I_{\lambda}=0$. Suppose that $p(x)q(x)r(x)=0$, where $p(x)=\sum^{m}_{i=0}a_{i}x^{i}$, $q(x)=\sum^{n}_{j=0}b_{j}x^{j}$ and $r(x)=\sum^{l}_{k=0}c_{k}x^{k}$ in $R[x;\sigma]$. Then $\bar{p}(x)\bar{q}(x)\bar{r}(x)=\bar{0}$ in $\big(R/I_{\lambda}\big)[x;\sigma_{\lambda}]$ for all $\lambda\in \Lambda$. Since $R/I_{\lambda}$ is  strongly $\sigma_{\lambda}$-symmetric for all $\lambda\in \Lambda$, we can deduce that $\bar{p}(x)\bar{r}(x)\bar{q}(x)=\bar{0}$. Then
$\sum_{t=s+j}\sum_{s=i+k} a_{i}\sigma^{i}(c_{k})\sigma^{s}(b_{j}) \in I_{\lambda}$ for all $\lambda\in \Lambda$. Hence,
$\sum_{t=s+j}\sum_{s=i+k} a_{i}\sigma^{i}(c_{k})\sigma^{s}(b_{j})=0$ since $\cap_{\lambda\in \Lambda}I_{\lambda}=0$. Therefore, $p(x)r(x)q(x)=0$ as required.
\end{proof}

Let $\sigma_{\gamma}$ be an endomorphism of a ring $R_{\gamma}$ for each $\gamma\in \Gamma$. Then the map $\sigma:\prod_{\gamma\in \Gamma}R_{\gamma}\rightarrow \prod_{\gamma\in \Gamma}R_{\gamma}$ defined by $\sigma((a_{\gamma}))=(\sigma_{\gamma}(a_{\gamma}))$ for $(a_{\gamma})\in \prod_{\gamma\in \Gamma}R_{\gamma}$ is endomorphism of $\prod_{\gamma\in \Gamma}R_{\gamma}$.
The proof of the following lemma is obtained by routine computations.

\begin{lemma}\label{product}
Let $R_{\gamma}$ be a ring with an endomorphism $\sigma_{\gamma}$ for each $\gamma\in \Gamma$. Then the followings are equivalent:
\begin{itemize}
  \item [(i)] $R_{\gamma}$ is strongly $\sigma_{\gamma}$-symmetric for each $\gamma\in \Gamma$.
  \item [(ii)] The direct product $\prod_{\gamma\in \Gamma}R_{\gamma}$ is strongly $\sigma$-symmetric.
  \item [(iii)] The direct sum $\bigoplus_{\gamma\in \Gamma}R_{\gamma}$ is strongly $\sigma$-symmetric.
\end{itemize}
\end{lemma}

A ring $R$ is called \emph{local} if $R/J(R)$ is a division ring, where $J(R)$ denotes the Jacobson radical of $R$. $R$ is called \emph{semilocal} if $R/J(R)$ is semisimple Artinian and $R$ is called \emph{semiperfect} if $R$ is semilocal and idempotents can be lifted modulo $J(R)$. Note that local rings are Abelian and semilocal (see \cite{L} for details).

\begin{proposition}
Let $R$ be a ring and $\sigma$ be an endomorphism of $R$. Then we have the followings.
\begin{itemize}
  \item [(i)]$R$ is a strongly $\sigma$-symmetric and semiperfect ring if and only if $R=\bigoplus^{n}_{i=1}R_{i}$ such that $R_{i}$ is local and strongly $\sigma_{i}$-symmetric ring, where $\sigma_{i}$ is an endomorphism of $R_{i}$ for all $i=0,1,\ldots n$.
  \item [(ii)] Let $R$ be a ring and $e$ be a central idempotent of $R$. Then $eR$ and $(1-e)R$ are strongly $\sigma$-symmetric if and only if $R$ is strongly $\sigma$-symmetric.
\end{itemize}
\end{proposition}
\begin{proof} (i) Suppose that $R$ is  strongly $\sigma$-symmetric  and semiperfect. Since $R$ is semiperfect, $R$ has a finite orthogonal set $\{e_1, e_2, \ldots, e_n \}$ of local idempotents whose sum is $1$ by \cite[Corollary 3.7.2]{JL}. Then $R=\sum\limits_{i=1}^ne_iR$ such that $e_{i}Re_{i}$ is a local ring for all $i=1,\ldots, n$. Since $R$ is strongly $\sigma$-symmetric, then $R$ is abelian and $e_iRe_i=e_iR$. Also, by Lemma~\ref{Lemma1-23.10.18}(iv), $\sigma(e_iR)\subseteq e_iR$ for all $i=1,\ldots, n$. Then $e_iR$ is strongly $\sigma_{i}$-symmetric and local subring of $R$, where $\sigma_{i}$ is an endomorphism of $e_iR$ induced by $\sigma$.
Conversely, let $R$ be a finite direct sum of strongly $\sigma_i$-symmetric local rings $R_{i}$ for all $i=0,1,\ldots n$. Then $R$ is semiperfect
since local rings are semiperfect and $R$ is  strongly $\sigma$-symmetric by Lemma~\ref{product}.\\
(ii) The proof is clear by Lemma~\ref{product} since $R \cong eR \oplus (1-e)R$.
\end{proof}


\section{Extensions of strongly $\sigma$-symmetric rings}

In this section, we investigate properties of strongly $\sigma$-symmetric rings and their extensions. First, we deal with the polynomial extensions of strongly $\sigma$-symmetric rings. Recall that if $\sigma$ is an endomorphism of a ring $R$, then the map $\sigma$ can be extended to an endomorphism of the polynomial ring $R[x]$ and the map $\bar{\sigma}:R[x]\rightarrow R[x]$ is defined by
$$\bar{\sigma}\left( \sum^{m}_{i=0}a_{i}x^{i}\right)=\sum^{m}_{i=0}\sigma (a_{i})x^{i}.$$
Polynomial rings over commutative (resp., reduced) rings are commutative (resp., reduced) obviously. On the other hand, Huh et. al. showed that polynomial rings over IFP rings need not be IFP \cite[Example 2]{HLS} and Kim and Lee proved that polynomial rings over reversible rings need not be reversible \cite[Example 2.1]{KL}. In \cite[Example 3.1]{HKKL}, it is also proved that polynomial rings over symmetric rings need not be symmetric. Based
on these results, one may ask whether polynomial rings over strongly $\sigma$-symmetric rings are also strongly $\bar{\sigma}$-symmetric. We remark that the idea of the following proof is similar to \cite[Theorem 6]{HKK}.

\begin{theorem} \label{rx}
Let $R$ be a ring with an endomorphism $\sigma$ such that $\sigma^{t}=1_{R}$ for some positive integer $t$, where $1_{R}$ denotes the identity endomorphism of $R$. Then $R$ is strongly $\sigma$-symmetric if and only if $R[x]$ is strongly $\bar{\sigma}$-symmetric.
\end{theorem}
\begin{proof}
It is enough to show that $R[x]$ is strongly $\bar{\sigma}$-symmetric when $R$ is strongly $\sigma$-symmetric. Assume that
\begin{align}
p(y)=&p_{0}+p_{1}y+\cdots +p_{m}y^{m} \notag , \\
q(y)=&q_{0}+q_{1}y+\cdots +q_{n}y^{n} \notag ,\\
r(y)=&r_{0}+r_{1}y+\cdots +r_{l}y^{l} \notag
\end{align}
in $R[x][y;\bar{\sigma}]$ such that $p(y)q(y)r(y)=0$. We also let
\begin{align}
p_{i}=&a_{i_{0}}+a_{i_{1}}x+\cdots +a_{i_{u_{i}}}x^{u_{i}} \notag ,\\
q_{j}=&b_{j_0}+b_{j_1}x+\cdots +b_{j_{v_{j}}}x^{v_{j}} \notag ,\\
r_{k}=&c_{k_0}+c_{k_1}x+\cdots +c_{k_{w_{k}}}x^{w_{k}} \notag
\end{align}
in $R[x]$, where $u_{i}, v_{j}, w_{k}\geq 0$ for each $0\leq i\leq m,~ 0\leq j\leq n$ and $0\leq k\leq l$. We claim that $p(y)r(y)q(y)=0$. Take a positive integer $s$ such that
\begin{center}
$s> max \{deg~ p_i,deg~ q_j, deg~ r_k\}$
\end{center}
for any $i,j$ and $k$, where the degree is considered as polynomials in $R[x]$ and the degree of zero polynomial is zero.
Let
\begin{align}
p(x^{ts+1})=&p_{0}+p_{1}x^{ts+1}+\cdots +p_{m}x^{mts+m} \notag, \\
q(x^{ts+1})=&q_{0}+q_{1}x^{ts+1}+\cdots +q_{n}x^{nts+n} \notag, \\
r(x^{ts+1})=&r_{0}+r_{1}x^{ts+1}+\cdots +r_{l}x^{lts+l} \notag.
\end{align}
Then the set of coefficients of the $p_i$(resp., $q_j$ and $h_k$) equals the set of coefficients of $p(x^{ts+1})$ (resp., $q(x^{ts+1})$ and $r(x^{ts+1})$). Since $p(y)q(y)r(y)=0$, $x$ commutes with elements of $R$ in the polynomial ring $R[x]$ and $\sigma^{ts}=1_R$, we get $p(x^{ts+1})q(x^{ts+1})r(x^{ts+1})=0 \in R[x,\sigma]$. Then $p(x^{ts+1})r(x^{ts+1})q(x^{ts+1})=0 \in R[x,\sigma]$ since $R$ is strongly $\sigma$-symmetric. Thus, we obtain $p(y)r(y)q(y)=0$ as required.
\end{proof}

We note that there is a non-identity endomorphism $\sigma$ of a strongly $\sigma$-symmetric ring $R$ such that $\sigma^{t}=I_{R}$ for some positive integer $t$ by the following example.

\begin{example} \label{nonidentity}\rm
Let $\mathbb{Z}$ be the ring of integers and $\mathbb{Z}_{4}$ be the ring of integers modulo 4. Consider the ring
$$R=\left\{ \left(\begin{array}{cc} a & \bar{b} \\ 0 & a \\ \end{array} \right)\mid a\in \mathbb{Z},\ \bar{b}\in \mathbb{Z}_{4}\right\}$$
and let $\sigma$ be an endomorphism defined by
$\sigma\left(\left(\begin{array}{cc} a & \bar{b}\\ 0 & a \\\end{array}\right)\right)=\left(\begin{array}{cc} a & -\bar{b}\\0 & a\\\end{array}\right)$. By \cite[Example 1.10]{HKR}, $R$ is $\sigma$-Armendariz and we also have $R$ is symmetric since it is commutative. Then by Theorem~\ref{sigmaA}, we obtain $R$ is strongly $\sigma$-symmetric and $\sigma^2=I_{R}$.
\end{example}

Recall that an element $u$ of a ring $R$ is {\emph right regular} if $ur = 0$ implies $r = 0$ for $r\in R$. Left regular elements can be defined similarly. An element is called \emph{regular} if it is both left and right regular.

A ring $R$ is called \emph{left Ore} for given $a, b\in R$ with $b$ regular, there exist $a_1, b_1\in R$ with $b_1$ regular such that $b_1a=a_1b$.
Also recall that the classical left quotient ring $S^{-1}R$ exists iff $S$ is a left Ore set and the set $\bar{S} = \{s + ass(S) \in R/ass(S)\ |\  s \in S\}$ consists
of regular elements (\cite{MR}, Theorem 2.1.12), where $ass(S) := \{r \in  R \ |\  sr = 0 \text{ for some } s \in S\}$. In \cite[Theorem 4.1]{HKKL}), it is proved that $R$ is symmetric if and only if $Q(R)$ is symmetric. In the following theorem, we consider the classical left quotient rings of strongly $\sigma$-symmetric rings.

Let $R$ be a ring with the classical left quotient ring $Q(R)$. Then each automorphism $\sigma$ of $R$ can be extended to $Q(R)$ by setting $\bar\sigma(b^{-1}a)=\sigma(b)^{-1}\sigma(a)$, where $a, b\in R$ with $b$ regular.

\begin{theorem} \label{Quotientring}
Let $R$ be a   ring with an automorphism $\sigma$. If the classical left quotient ring $Q(R)$ of $R$ exists, then $R$ is strongly $\sigma$-symmetric if and only if $Q(R)$ is strongly $\bar\sigma$-symmetric.
\end{theorem}
\begin{proof} It is enough to show that $Q(R)$ is strongly $\bar\sigma$-symmetric whenever $R$ is strongly $\sigma$-symmetric. Let $p(x)=\sum^{m}_{i=0}u_{1}^{-1}a_{i}x^{i}$, $q(x)=\sum^{n}_{j=0}v_{1}^{-1}b_{j}x^{j}$ and $r(x)=\sum^{l}_{k=0}w_{1}^{-1}c_{k}x^{k}\in Q(R)[x;\bar\sigma]$ such that $p(x)q(x)r(x)=0$, where $a_i, b_j, c_k\in R$ and $u, v, w$ are regular elements in $R$ for $0\leq i\leq m$, $0\leq j\leq n$ and $0\leq k\leq l$. Then we obtain
\begin{align}
0=&p(x)q(x)r(x)\notag \\
 =&u_{1}^{-1}\bigg(\sum^{m}_{i=0}a_{i}x^{i}v^{-1}_{1} \bigg)\bigg(\sum^{n}_{j=0}b_{j}x^{j}w^{-1}_{1} \bigg)\bigg(\sum^{l}_{k=0}c_{k}x^{k} \bigg)\notag \\
 =&\bigg(\sum^{m}_{i=0}a_{i}\sigma^{i}(v_{1})^{-1}x^{i}\bigg)\bigg(\sum^{n}_{j=0}b_{j}\sigma^{j}(w_{1})^{-1}x^{j}\bigg)\bigg(\sum^{l}_{k=0}c_{k}x^{k}\bigg)\notag.
\end{align}
There exist $a'_{i}, b'_{j}\in R$ and regular elements $v_{2}, w_{2}\in R$ such that
\begin{equation}
a_{i}\sigma^{i}(v_{1})^{-1}=v_{2}^{-1}a'_{i},\label{ai}
\end{equation}
\begin{equation}
b_{j}\sigma^{j}(w_{1})^{-1}=w_{2}^{-1}b'_{j}\label{bj}
\end{equation}
for $0\leq i\leq m$ and $0\leq j\leq n$. Thus we have
\begin{align}
0=&\bigg(\sum^{m}_{i=0}v_{2}^{-1}a'_{i}x^{i}\bigg)
             \bigg(\sum^{n}_{j=0}w_{2}^{-1}b'_{j}x^{j}\bigg)
             \bigg(\sum^{l}_{k=0}c_{k}x^{k}\bigg)\notag \\
 =&v_{2}^{-1}\bigg(\sum^{m}_{i=0}a'_{i}\sigma^{i}(w_{2})^{-1}x^{i}\bigg)\bigg(\sum^{n}_{j=0}b'_{j}x^{j}\bigg)\bigg(\sum^{l}_{k=0}c_{k}x^{k}\bigg).\notag
\end{align}
There exist $a''_{i}\in R$ and regular element $w_{3}\in R$ such that
\begin{equation}
a'_{i}\sigma^{i}(w_{2})^{-1}=w_{3}^{-1}a''_{i}\label{aii}
\end{equation}
for $0\leq i\leq m$. Then we have
$\bigg(\sum^{m}_{i=0}a''_{i}x^{i}\bigg)\bigg(\sum^{n}_{j=0}b'_{j}x^{j}\bigg)\bigg(\sum^{l}_{k=0}c_{k}x^{k}\bigg)=0.$ Using strongly $\sigma$-symmetric property of $R$, we can deduce that
\begin{equation}
\bigg(\sum^{m}_{i=0}a''_{i}x^{i}\bigg)\bigg(\sum^{l}_{k=0}c_{k}x^{k}\bigg)\bigg(\sum^{n}_{j=0}b'_{j}x^{j}\bigg)=0.
\end{equation}
Since $R$ is strongly $\sigma$-IFP, we have
\begin{align}
0=&\bigg(\sum^{m}_{i=0}a''_{i}x^{i}\bigg)w_{2}v_{1}\bigg(\sum^{l}_{k=0}c_{k}x^{k}\bigg)\bigg(\sum^{n}_{j=0}b'_{j}x^{j}\bigg)\notag \\
 =&\bigg(\sum^{m}_{i=0}a''_{i}\sigma^{i}(w_{2}v_{1})x^{i}\bigg)\bigg(\sum^{l}_{k=0}c_{k}x^{k}\bigg)\bigg(\sum^{n}_{j=0}b'_{j}x^{j}\bigg).\label{prq}
\end{align}
By multiplying (\ref{prq}) on the left hand side by $(w_{3}v_{2})^{-1}$ and by using (\ref{ai}) and (\ref{aii}), we obtain
\begin{align}
0=&\bigg(\sum^{m}_{i=0}(w_{3}v_{2})^{-1}a''_{i}\sigma^{i}(w_{2}v_{1})x^{i}\bigg)\bigg(\sum^{l}_{k=0}c_{k}x^{k}\bigg)\bigg(\sum^{n}_{j=0}b'_{j}x^{j}\bigg)\notag \\
 =&\bigg(\sum^{m}_{i=0}a_{i}x^{i}\bigg)\bigg(\sum^{l}_{k=0}c_{k}x^{k}\bigg)\bigg(\sum^{n}_{j=0}b'_{j}x^{j}\bigg)\notag.
\end{align}
Thus we get
\begin{equation}\label{prq1}
\bigg(\sum^{n}_{j=0}b'_{j}x^{j}\bigg)w_{1}\bigg(\sum^{m}_{i=0}a_{i}x^{i}\bigg)\bigg(\sum^{l}_{k=0}c_{k}x^{k}\bigg)=0
\end{equation}
since $R$ is strongly $\sigma$-skew reversible and strongly $\sigma$-IFP.
If we multiply (\ref{prq1}) by $w_{2}^{-1}$ on the left hand side and use (\ref{bj}), then we have
\begin{align}
0=&w_{2}^{-1}\bigg(\sum^{n}_{j=0}b'_{j}\sigma^{j}(w_{1})x^{j}\bigg)\bigg(\sum^{m}_{i=0}a_{i}x^{i}\bigg)\bigg(\sum^{l}_{k=0}c_{k}x^{k}\bigg)\notag \\
 =&\bigg(\sum^{n}_{j=0}b_{j}x^{j}\bigg)\bigg(\sum^{m}_{i=0}a_{i}x^{i}\bigg)\bigg(\sum^{l}_{k=0}c_{k}x^{k}\bigg)\notag \\
 =&\bigg(\sum^{m}_{i=0}a_{i}x^{i}\bigg)\bigg(\sum^{l}_{k=0}c_{k}x^{k}\bigg)\bigg(\sum^{n}_{j=0}b_{j}x^{j}\bigg).\label{lasteq}
\end{align}
Therefore, we get $p(x)r(x)q(x)=0$ from (\ref{lasteq}) by using similar arguments above. Hence, $Q(R)$ is strongly $\bar\sigma$-symmetric.
\end{proof}


Let $R$ be a ring with an automorphism $\sigma$ and suppose that $\Delta$ is a multiplicatively closed subset of $R$ consisting of all central regular elements. Then the map $\bar{\sigma}:\Delta^{-1}R\rightarrow \Delta^{-1}R$ defined by $\bar{\sigma}(u^{-1}r)=\sigma(u)^{-1}\sigma(r)$ is also an automorphism, where $r\in R$ and $u\in R$ is a regular element.

Note that, in the following proposition, we obtain the results proved in \cite[Proposition 3.6]{BHKKL} and \cite[Proposition 3.8]{JKKL}  without the condition '$\sigma(u)=u$' for any central regular element $u$. Therefore, we obtain a generalization of \cite[Lemma 3.2]{HKKL} without any condition on $\sigma$.

\begin{proposition} \label{delta}
Let $R$ be a ring with an automorphism $\sigma$.  Then $R$ is strongly $\sigma$-symmetric if and only if $\Delta^{-1}R$ is strongly $\bar{\sigma}$-symmetric.
\end{proposition}
\begin{proof}
Let $p(x)q(x)r(x)=0$, where $p(x)=\sum^{m}_{i=0}u^{-1}a_ix^{i}$, $q(x)=\sum^{n}_{j=0}v^{-1}b_jx^{j}$ and $r(x)=\sum^{l}_{k=0}w^{-1}c_kx^{k}\in \Delta^{-1}R[x;\bar\sigma]$ with  $a_i, b_j, c_k\in R$ and $u, v, w$ are central regular elements in $R$ for all $0\leq i\leq m$,  $0\leq j\leq n$ and $0\leq k\leq l$. Then we have
\begin{align}
0=&p(x)q(x)r(x)\notag \\
 =&\bigg(\sum^{m}_{i=0}u^{-1}a_ix^{i}\bigg)\bigg(\sum^{n}_{j=0}v^{-1}b_jx^{j}\bigg)\bigg(\sum^{l}_{k=0}w^{-1}c_kx^{k}\bigg)\notag \\
 =&u^{-1}\bigg(\sum^{m}_{i=0}a_{i}\sigma^{i}(v)^{-1}x^{i} \bigg)\bigg(\sum^{n}_{j=0}b_{j}\sigma^{j}(w)^{-1}x^{j} \bigg)\bigg(\sum^{l}_{k=0}c_{k}x^{k} \bigg)\notag \\
 =&u^{-1}\bigg(\sum^{m}_{i=0}\sigma^{i}(v)^{-1}a_{i}x^{i}\bigg)\bigg(\sum^{n}_{j=0}\sigma^{j}(w)^{-1}b_{j}x^{j}\bigg)\bigg(\sum^{l}_{k=0}c_{k}x^{k}\bigg)\notag.
\end{align}
There exist $a'_i, b'_j\in R$ and central regular elements $v_1, w_1\in R$ such that
\begin{equation}
\sigma^{i}(v)^{-1}a_{i}=v^{-1}_1a'_i \label{centraleq1}
\end{equation}
\begin{equation}
\sigma^{j}(w)^{-1}b_{j}=w^{-1}_1b'_j \label{centraleq2}
\end{equation}
for all $0\leq i\leq m$ and $0\leq j\leq n$. Then we have
\begin{align}
0=&u^{-1}v^{-1}_1\bigg(\sum^{m}_{i=0}a'_{i}\sigma^{i}(w_1)^{-1}x^{i}\bigg)\bigg(\sum^{n}_{j=0}b'_{j}x^{j}\bigg)\bigg(\sum^{l}_{k=0}c_{k}x^{k}\bigg)\notag \\
 =&u^{-1}v^{-1}_1\bigg(\sum^{m}_{i=0}\sigma^{i}(w_1)^{-1}a'_{i}x^{i}\bigg)\bigg(\sum^{n}_{j=0}b'_{j}x^{j}\bigg)\bigg(\sum^{l}_{k=0}c_{k}x^{k}\bigg)\notag.
\end{align}
There exist $a''_{i}\in R$ and central regular element $w_2\in R$ such that
\begin{equation}
\sigma^{i}(w_1)^{-1}a'_{i}=w^{-1}_2a''_{i}\label{centraleq3}
\end{equation}
for all $0\leq i\leq m$. Then we obtain
$$
0=u^{-1}v^{-1}_1w^{-1}_2\bigg(\sum^{m}_{i=0}a''_{i}x^{i}\bigg)\bigg(\sum^{n}_{j=0}b'_{j}x^{j}\bigg)\bigg(\sum^{l}_{k=0}c_{k}x^{k}\bigg).\notag \\
$$
Hence, we get
$\bigg(\sum^{m}_{i=0}a''_{i}x^{i}\bigg)\bigg(\sum^{n}_{j=0}b'_{j}x^{j}\bigg)\bigg(\sum^{l}_{k=0}c_{k}x^{k}\bigg)=0.$
Since $R$ is strongly $\sigma$-symmetric, we can deduce that
\begin{align*}
0=&\bigg(\sum^{m}_{i=0}a''_{i}x^{i}\bigg)\bigg(\sum^{l}_{k=0}c_{k}x^{k}\bigg)\bigg(\sum^{n}_{j=0}b'_{j}x^{j}\bigg).
\end{align*}
By using the fact that $R$ is strongly $\sigma$-IFP, we have \begin{align}
0=&\bigg(\sum^{m}_{i=0}a''_{i}x^{i}\bigg)vw_1\bigg(\sum^{l}_{k=0}c_{k}x^{k}\bigg)\bigg(\sum^{n}_{j=0}b'_{j}x^{j}\bigg).\label{delta1}
\end{align}
If we multiply (\ref{delta1}) by $(w_2v_1)^{-1}$ on the left hand side and if we use (\ref{centraleq1}) and (\ref{centraleq3}), we get
\begin{align}
0=&(w_2v_1)^{-1}\bigg(\sum^{m}_{i=0}a''_{i}\sigma^{i}(vw_1)x^{i}\bigg)\bigg(\sum^{l}_{k=0}c_{k}x^{k}\bigg)\bigg(\sum^{n}_{j=0}b'_{j}x^{j}\bigg)\notag\\
 =&\bigg(\sum^{m}_{i=0}a_{i}x^{i}\bigg)\bigg(\sum^{l}_{k=0}c_{k}x^{k}\bigg)\bigg(\sum^{n}_{j=0}b'_{j}x^{j}\bigg).\notag
\end{align}
Since $R$ is strongly $\sigma$-skew reversible and strongly $\sigma$-IFP, we obtain that
\begin{align}
0=&\bigg(\sum^{n}_{j=0}b'_{j}x^{j}\bigg)w\bigg(\sum^{m}_{i=0}a_{i}x^{i}\bigg)\bigg(\sum^{l}_{k=0}c_{k}x^{k}\bigg).\label{delta2}
\end{align}
By multiplying (\ref{delta2}) with $w^{-1}_1$ on the left hand side and using (\ref{centraleq2}), we get
\begin{align}
0=&w^{-1}_1\bigg(\sum^{n}_{j=0}b'_{j}x^{j}\bigg)w\bigg(\sum^{m}_{i=0}a_{i}x^{i}\bigg)\bigg(\sum^{l}_{k=0}c_{k}x^{k}\bigg)\notag \\
 =&\bigg(\sum^{n}_{j=0}w^{-1}_1b'_{j}\sigma^{j}(w)x^{j}\bigg)\bigg(\sum^{m}_{i=0}a_{i}x^{i}\bigg)\bigg(\sum^{l}_{k=0}c_{k}x^{k}\bigg)\notag \\
 =&\bigg(\sum^{n}_{j=0}b_{j}x^{j}\bigg)\bigg(\sum^{m}_{i=0}a_{i}x^{i}\bigg)\bigg(\sum^{l}_{k=0}c_{k}x^{k}\bigg).\notag
\end{align}
Then we have $\bigg(\sum^{m}_{i=0}a_{i}x^{i}\bigg)\bigg(\sum^{l}_{k=0}c_{k}x^{k}\bigg)\bigg(\sum^{n}_{j=0}b_{j}x^{j}\bigg)=0$. This result leads us  $p(x)r(x)q(x)=0$ by using similar arguments above. Therefore, $\Delta^{-1}R$ is strongly $\bar\sigma$-symmetric.
\end{proof}

The ring of Laurent polynomials in $x$ over $R$ is denoted by $R[x,x^{-1}]$ and consists of all formal sums $\sum^{n}_{i=k}a_{i}x^{i}$ with obvious addition and multiplication, where $a_{i}\in R$ and $k, n$ are (possibly negative) integers. Note that the map $\bar{\sigma}:R[x,x^{-1}]\rightarrow R[x,x^{-1}]$ defined by $\bar{\sigma}(\sum^{n}_{i=k}a_{i}x^{i})=\sum^{n}_{i=k}\sigma(a_{i})x^{i}$ is an endomorphism of $R[x,x^{-1}]$ which extends $\sigma$.

\begin{corollary}
Let $R$ be a ring and $\sigma$ be an endomorphism of $R$. Then $R[x]$ is strongly $\bar{\sigma}$-symmetric if and only if so is $R[x,x^{-1}]$.
\end{corollary}
\begin{proof}
Let $\Delta=\{1, x, x^{2}, \cdots\}$. Then clearly $\Delta$ is a multiplicatively closed subset of $R[x]$ and $R[x,x^{-1}]=\Delta^{-1}R[x]$. By the Proposition \ref{delta}, it follows that $R[x, x^{-1}]$ is strongly $\bar{\sigma}$-symmetric.
\end{proof}

Therefore, we obtain generalizations of the following results.

\begin{corollary}(\cite[Lemma 3.2]{HKKL}) (1) Let $R$ be a ring and $\Delta$ be a multiplicatively closed subset of $R$ consisting of central regular elements. Then $R$ is symmetric if and only if so is $\Delta^{-1}R$.\\
(2) For a ring $R$, $R[x]$ is symmetric if and only if so is $R[x;x^{-1}]$.
\end{corollary}

Let $R$ be a ring and $\sigma$ be a monomorphism of $R$. We consider the Jordan's construction of $R$ by $\sigma$, which is the minimal extension of $R$ to which $\sigma$ extends  as an automorphism. Let $A(R,\sigma)$ be the subset $\{x^{-i}rx^{i}\vert \ r\in R \mbox{ and } i\geq 0\}$ of the skew Laurent polynomial ring $R[x,x^{-1};\sigma]$, where $\sigma$ is a monomorphism of $R$.
Multiplication is subject to $xr=\sigma(r)x$ and $rx^{-1}=x^{-1}\sigma(r)$ for all $r\in R$. Also note that $x^{-i}rx^{i}=x^{-(i+j)}\sigma^{j}(r)x^{i+j}$ for each $j\geq 0$. It follows that $A(R,\sigma)$ forms a subring of $R[x,x^{-1};\sigma]$ with the following operations: $x^{-i}rx^{i}+x^{-j}sx^{j}=x^{-(i+j)}(\sigma^{j}(r)+\sigma^{i}(s))x^{(i+j)}$ and $(x^{-i}rx^{i})(x^{-j}sx^{j})=x^{-(i+j)}\sigma^{j}(r)\sigma^{i}(s)x^{(i+j)}$ for $r,s \in R$ and $i,j\geq 0$. Furthermore, $A(R,\sigma)$ is an extension of $R$ by $\sigma$ and the map $\sigma$ can be extended to an automorphism $\bar{\sigma}$ of $A(R,\sigma)$ defined by $\bar{\sigma}(x^{-i}rx^{i})=x^{-i}\sigma(r)x^{i}$. In \cite{J}, Jordan proved that for any pair of $(R,\sigma)$ such an extension always exists. $A(R,\sigma)$ is called {\em Jordan extension of $R$ by $\sigma$}.

\begin{proposition}\label{Jordanext}
Let $R$ be a ring and $\sigma$ be a monomorphism of the ring $R$. Then $R$ is strongly $\sigma$-symmetric if and only if $A(R,\sigma)$ is a strongly $\bar{\sigma}$-symmetric.
\end{proposition}
\begin{proof}
Suppose that $R$ is strongly $\sigma$-symmetric and let $p(y)q(y)r(y)=0$, where $p(y)=\sum^{m}_{i=0}a_{i}y^{i}, q(y)=\sum^{n}_{j=0}b_{j}y^{j}$ and $r(y)=\sum^{l}_{k=0}c_{k}y^{k}$ in $A(R,\sigma)[y;\bar{\sigma}]$ such that $a_{i}=x^{-u_i}a'_{i}x^{u_i}$, $b_{j}=x^{-v_j}b'_{j}x^{v_j}$, $c_{k}=x^{-w_k}c'_{k}x^{w_k}$ for $a'_{i}, b'_{j}, c'_{k} \in R$ and $u_i, v_j,w_k\geq 0$ for all $0\leq i\leq m$, $0\leq j\leq n$ and $0\leq k\leq l$. Then we have
$$a_{i}=x^{-\mu}a''_{i}x^{\mu},~~b_{j}=x^{-\mu}b''_{j}x^{\mu}~~\mbox{and}~~c_{k}=x^{-\mu}c''_{k}x^{\mu}$$
for some $\mu\geq 0$, where $a''_{i}, b''_{j}, c''_{k} \in R$ for all $0\leq i\leq m$, $0\leq j\leq n$ and $0\leq k\leq l$. Since $p(y)q(y)r(y)=0$, we have
$$\sum^{m+n+l}_{t=0} \bigg( \sum_{t=s+k}\bigg(\sum_{s=i+j} a_{i}\bar\sigma^{i}(b_{j})\bar\sigma^{s}(c_{k})\bigg) \bigg) y^{t}=0.$$
Thus
\begin{equation}\label{jordan1}
\sum_{t=s+k}\sum_{s=i+j} a_{i}\bar\sigma^{i}(b_{j})\bar\sigma^{s}(c_{k})=\sum_{t=s+k}\sum_{s=i+j} x^{-\mu}a''_{i}\sigma^{i}(b''_{j})\sigma^{s}(c''_{k})x^{\mu}=0.
\end{equation}
Let $p'(y)=\sum^{m}_{i=0}a''_{i}y^{i},~q'(y)=\sum^{n}_{j=0}b''_{j}y^{j}$ and $r'(y)=\sum^{l}_{k=0}c''_{k}y^{k}\in R[x;\sigma]$. Then by (\ref{jordan1}), we get $p'(y)q'(y)r'(y)=0$ and so $p'(y)r'(y)q'(y)=0$ since $R$ is strongly $\sigma$-symmetric. Therefore, $p(y)r(y)q(y)=0$. The converse is obvious since $R$ is a subring of a strongly $\bar{\sigma}$-symmetric ring.
\end{proof}

Let $R$ be a ring with an ideal $I$ and $\sigma$ be an endomorphism of $R$. If $I$ is a $\sigma$-ideal of $R$ (i.e., $\sigma(I)\subseteq I$), then $\bar{\sigma}:R/I\rightarrow R/I$ defined by $\bar{\sigma}(a+I)=\sigma(a)+I$ for $a\in R$ is an endomorphism of $R/I$. Note that $R/I$ need not be a strongly $\bar{\sigma}$-symmetric ring for every ideal $I$ of a strongly $\sigma$-symmetric ring $R$. Indeed, if $R$ is the ring of quaternions with integer coefficients and $\sigma$ is a monomorphism of $R$, then $R$ is a domain and so strongly $\sigma$-symmetric; while for any odd prime integer $q$, we have $R/qR\cong Mat_{2}(\mathbb{Z}_{q})$ by the arguments in \cite[Exercise 3A]{GW}. Notice that $Mat_{2}(\mathbb{Z}_{q})$ is not strongly $\bar{\sigma}$-symmetric since it is not abelian and thus the factor ring $R/qR$ is not strongly $\bar{\sigma}$-symmetric.
But we have an affirmative answer for the following situation.

Recall first that for a subset $S$ of a ring $R$, the set $r_{R}(S)=\{c\in R\mid Sc=0\}$~(resp., $l_{R}(S)=\{c\in R\mid cS=0\}$) is called the \emph{right} (resp., \emph{left}) \emph{annihilator} of $S$ in $R$.

\begin{proposition}\label{factor}
Let $R$ be ring and $\sigma$ be an endomorphism of $R$. Suppose that $R$ is a strongly $\sigma$-symmetric ring. If $I$ is a one-sided annihilator in $R$ and $\sigma(I)\subseteq I$, then $R/I$ is strongly $\bar{\sigma}$-symmetric.
\end{proposition}
\begin{proof}
Set $I=r_R(S)$ for some $S\subseteq R$. We write $\bar R=R/I$ and $\bar r=r+I$ for $r\in R$.
We have $R$ is symmetric and so has IFP since $R$ is strongly $\sigma$-symmetric. By \cite[Lemma 1.2]{S}, $I$ is an ideal of $R$. Let $\bar p(x)=\sum\limits_{i=0}^m \bar a_i x^i$, $\bar q(x)= \sum\limits_{j=0}^n\bar b_jx^j$ and $\bar r(x)=\sum\limits_{k=0}^l \bar c_k x^k\in \bar R[x;\bar\sigma]$ with $\bar p(x) \bar
q(x) \bar r(x)= \bar 0$. Then $p(x)q(x)r(x)\in I[x;\sigma]$ and hence, $Sp(x)q(x)r(x)=0$. Thus $Sp(x)r(x)q(x)=0$ since $R$ is strongly $\sigma$-symmetric and this implies that $\bar p(x) \bar r(x) \bar q(x) = \bar 0$. Therefore, $R/I$ is strongly $\bar\sigma$-symmetric.
The left annihilator case can be proved similarly.
\end{proof}

As a kind of converse of Proposition \ref{factor}, we obtain the following result.

\begin{proposition}\label{factor2}
Let $R$ be a ring with an endomorphism $\sigma$ and $I$ be a $\sigma$-ideal of $R$. If $R/I$ is a strongly $\bar\sigma$-symmetric ring and $I$ is a $\sigma$-rigid ring without identity, then $R$ is strongly $\sigma$-symmetric.
\end{proposition}
\begin{proof}
Let $p(x)q(x)r(x)=0$ for $p(x), q(x), r(x)\in R[x;\sigma]$. Then  $\bar p(x) \bar q(x) \bar r(x)= \bar 0$ and since $R/I$ is strongly $\bar{\sigma}$-symmetric, we have $\bar p(x) \bar r(x) \bar q(x)= \bar 0$. Thus $p(x)r(x)q(x) \in I[x,\sigma]$. By \cite[Proposition 3]{HKK}, we have that $I[x;\sigma]$ is reduced and hence, symmetric. Also, by \cite[Proposition 1]{L}, we get that
all the possible products of $p(x), q(x)$ and $r(x)$ is in $I[x;\sigma]$.
Then we obtain
$$(q(x)r(x)p(x))^{2}=q(x)r(x)[p(x)q(x)r(x)]p(x)=0$$
and so $q(x)r(x)p(x)=0$ since $I[x;\sigma]$ is reduced. Thus
\begin{align}
0=&p(x)r(x)[q(x)r(x)p(x)]r(x)q(x)\notag\\
 =&[p(x)r(x)q(x)][r(x)p(x)r(x)q(x)]\notag\\
 =&[r(x)p(x)r(x)q(x)][p(x)r(x)q(x)]\label{2}
\end{align}
by using the fact that $I[x;\sigma]$ is symmetric and so reversible.
If we multiply (\ref{2}) on the right hand side by $p(x)$, then we obtain
\begin{align}
0=&r(x)[p(x)r(x)q(x)p(x)r(x)q(x)p(x)]\notag\\
 =&[p(x)r(x)q(x)p(x)r(x)q(x)p(x)]r(x) \label{5}.
\end{align}
If we multiply (\ref{5}) on the right hand side by $q(x)$, then we get $(p(x)r(x)q(x))^{3}=0$ and hence, $p(x)r(x)q(x)=0$. Therefore, $R$ is strongly $\sigma$-symmetric.
\end{proof}

As a consequence of Propositon~\ref{factor} and Proposition~\ref{factor2}, we can give the following corollary.

\begin{corollary}
(1) \cite[Proposition 3.5]{HKKL} Let $R$ be a symmetric ring and $I$ be an ideal of $R$. If $I$ is an annihilator in $R$, then $R/I$ is symmetric.\\
(2) \cite[Proposition 3.6(1)]{HKKL} Let $R$ be a ring and $I$ be a proper ideal of $R$. If $R/I$ is symmetric and $I$ is reduced (as a ring without identity), then $R$ is symmetric.
\end{corollary}

Note that Proposition \ref{factor2} is false without the assumption ''$I$ is a $\sigma$-rigid ring without identity'' by the following example.

\begin{example}
{\rm  Consider a ring $R=U_2(F)$, where $F$ is a division ring and $\sigma$ is an automorphism of $R$
defined by  $\sigma \left( \left(\begin{array}{rr} a
& b\\0 & c\end{array} \right) \right) = \left(\begin{array}{rr}a & -b\\
0 & c\end{array} \right).$ Then  $R$ is not Abelian and so $R$ is not strongly $\sigma$-symmetric. The $\sigma$-ideal
$ I= \left(\begin{array}{rr}F& F\\0 & 0 \end{array} \right)$ of $R$ is not $\sigma$-rigid as a ring without identity. Indeed, $A\sigma(A)=0$, but $A\neq0$ for $A=\left( \begin{array}{rr}0& 1\\0 & 0\end{array} \right) \in I$. Also, note that the factor ring $R/I\cong F$ is reduced and $\bar{\sigma}$ is the identity map on $R/I$. Thus $R/I$ is  $\bar{\sigma}$-rigid and hence, strongly $\bar{\sigma}$-symmetric.  }
\end{example}

Let $R$ be an algebra over a commutative ring $S$. Following \cite{DO}, the Dorroh extension of $R$ by $S$ is the abelian group $D=R\oplus S$ with multiplication given by $(r_1,s_1)(r_2,s_2)=(r_1r_2+s_1r_2+s_2r_1,s_1s_2)$, where $r_{i}\in R$ and $s_{i}\in S$. For an $S$-algebra homomorphism $\sigma$ of $R$, $\sigma$ can be extended to an $S$-algebra homomorphism  $\bar{\sigma}:D\rightarrow D$ defined by $\bar{\sigma}((r,s))=(\sigma(r),s)$.

\begin{theorem}\label{Dorroh}
Let $R$ be an algebra over a commutative domain $S$ and $\sigma$ be an $S$-algebra homomorphism of $R$. Then $R$ is strongly $\sigma$-symmetric if and only if the Dorroh extension $D$ of $R$ by $S$ is strongly $\bar{\sigma}$-symmetric.
\end{theorem}
\begin{proof} It is enough to show that the Dorroh extension $D$ is strongly $\bar{\sigma}$-symmetric when $R$ is strongly $\sigma$-symmetric. Let $p(x)q(x)r(x)=0$, where $p(x)=\sum\limits_{i=0}^m (a_i,b_i)x^i$, $q(x)= \sum\limits_{j=0}^n(c_j,d_j)x^j$ and $r(x)=\sum\limits_{k=0}^l (e_k,f_k)x^k\in D[x;\bar\sigma]$. Then we have
\begin{align}\label{DorrohEq1}
0 = & \bigg( \sum\limits_{s=0}^{m+n}\bigg( \sum\limits_{s=i+j}(a_i,b_i)\bar{\sigma}^{i}(c_j,d_j)x^{s}\bigg)\bigg)\bigg(\sum\limits_{k=0}^{l}(e_k,f_k)x^k\bigg)\notag \\
=&\sum\limits_{t=0}^{m+n+l}\bigg(\sum\limits_{t=s+k}\bigg(\sum\limits_{s=i+j}(a_i,b_i)\bar{\sigma}^{i}(c_j,d_j)\bar{\sigma}^{s}(e_k,f_k)\bigg)\bigg)x^t \notag\\
= & \sum\limits_{t=0}^{m+n+l}\bigg(\sum\limits_{t=s+k}\bigg(\sum\limits_{s=i+j}(a_i,b_i)(\sigma^{i}(c_j),d_j)(\sigma^{s}(e_k),f_k)\bigg)\bigg)x^t.
\end{align}
Let $p'(x)=\sum\limits_{i=0}^{m}b_i x^i$, $q'(x)= \sum\limits_{j=0}^n d_j x^j$ and $r'(x)=\sum\limits_{k=0}^l f_k x^k \in S[x]$. Then by (\ref{DorrohEq1}), we get $p'(x)q'(x)r'(x)=0$. Since $S[x]$ is a domain, either $p'(x)=0$ or $q'(x)=0$ or $r'(x)=0$.
Let $p'(x)=0$. If we use the facts that $R$ is an $S$-algebra and $\sigma^{i}(s)=s$ for each $s\in S$ and $i\in \mathbb{N}$, then (\ref{DorrohEq1}) becomes
\begin{center}
$
\begin{array}
[c]{lll}%
0&= & \sum\limits_{t=0}^{m+n+l}\bigg(\sum\limits_{t=s+k}\bigg(\sum\limits_{s=i+j}(a_{i},0)(\sigma^{i}%
(c_{j}),d_{j})(\sigma^{s}(e_{k}),f_{k})\bigg)\bigg)x^{t}\\
&= & \sum\limits_{t=0}^{m+n+l}\bigg(\sum\limits_{t=s+k}\bigg(\sum\limits_{s=i+j}(a_{i}\sigma^{i}%
(c_{j})+d_{j}a_{i},0)(\sigma^{s}(e_{k}),f_{k})\bigg)\bigg)x^{t}\\
&= & \sum\limits_{t=0}^{m+n+l}\bigg(\sum\limits_{t=s+k}\bigg(\sum\limits_{s=i+j}(a_{i}\sigma^{i}%
(c_{j})\sigma^{s}(e_{k})+d_{j}a_{i}\sigma^{s}(e_{k})+f_{k}a_{i}\sigma
^{i}(c_{j})+f_{k}d_{j}a_{i},0)\bigg)\bigg)x^{t}\\
&= & \sum\limits_{t=0}^{m+n+l}\bigg(\sum\limits_{t=s+k}\bigg(\sum\limits_{s=i+j}(a_{i}\sigma^{i}%
(c_{j})\sigma^{s}(e_{k})+a_{i}\sigma^{i}(d_{j})\sigma^{s}(e_{k}%
)+a_{i}\sigma^{i}(c_{j})\sigma^{s}(f_{k})+a_{i}\sigma^{i}(d_{j})\sigma^{s}(f_{k}),0)\bigg)\bigg)x^{t}.
\end{array}
$
\end{center}
Let $p''(x)=\sum\limits_{i=0}^{m}a_i x^i$, $q''(x)= \sum\limits_{j=0}^n (c_{j}+d_j)x^j$ and $r''(x)=\sum\limits_{k=0}^l (e_k+f_k)x^k \in R[x;\sigma]$. Then $p''(x)q''(x)r''(x)=0$. Since $R$ is strongly $\sigma$-symmetric, we have $p''(x)r''(x)q''(x)=0$ and this implies that $p(x)r(x)q(x)=0$. For the cases $q'(x)=0$ and $r'(x)=0$, the proof can be seen by using similar arguments.
\end{proof}

By Theorem~\ref{Dorroh}, we obtain a generalization of the following corollary.

\begin{corollary}(\cite[Proposition 4.2 (2)]{HKKL})
Let $R$ be an algebra over a commutative ring $S$, and $D$ be the Dorroh extension of $R$ by $S$. If $R$ is symmetric and $S$ is a domain, then $D$ is symmetric.
\end{corollary}

In the following result, we give a criteria for transfer strongly $\sigma$-symmetric property from one ring to another.

\begin{proposition}
Let $\alpha:R\rightarrow S$ be a ring isomorphism. Then $R$ is a strongly $\sigma$-symmetric ring if and only if $S$ is a strongly $\alpha\sigma\alpha^{-1}$-symmetric ring.
\end{proposition}
\begin{proof}
It follows from the fact that an isomorphism $\alpha:R\rightarrow S$ induces an isomorphism from $R[x;\sigma]$ to $S[x;\alpha\sigma\alpha^{-1}]$.
\end{proof}

We denote the $n\times n$ full matrix ring over $R$ (resp., upper triangular matrix ring) by $Mat_{n}(R)$~(resp., $U_{n}(R)$) for $n\geq 2$. Also consider the following subrings of $Mat_n(R)$
$$D_n(R)=\left \{\left( \begin{array}{ccccc} a & a_{12} &a_{13}&\cdots& a_{1n}\\
0 & a&a_{23}&\cdots &a_{2n}\\ 0 & 0&a&\cdots &a_{3n}\\ \vdots
&\vdots&\vdots&\ddots&\vdots\\ 0&0&0&\cdots &a\end{array}  \right)
\mid a, a_{ij}\in R \text{ , } i=1,\ldots,n \text{ and } j=2,\ldots, n \right\}$$ and  $$V_n(R)=\{(a_{ij})\in
D_n(R)\mid a_{ij}=a_{(i+1)(j+1)}\text{ for }i=1, \ldots, n-2\text{
and }j=2, \ldots, n-1\}.$$ Note that $V_n(R) \cong R[x]/(x^n)$, where  $(x^n)$  is the ideal  of $R[x]$ generated by $x^n$.
Let $\sigma$ be an endomorphism of $R$. Then $\sigma$ can be extended to an endomorphism of $D_{n}(R)$ (resp., $V_{n}(R)$) defined by $\bar{\sigma}((a_{ij}))=(\sigma(a_{ij}))$. We use the same notation $\bar \sigma$ for the extension endomorphism of $D_{n}(R)$ and $V_{n}(R)$.

It is known that $Mat_{n}(R)$ and $U_{n}(R)$ do not have IFP for $n\geq 2$ since $U_2(R)$ is not Abelian. Also note that by \cite[Example 1.3]{KL}, $D_{n}(R)$ does not have IFP for $n\geq 4$. Thus $Mat_{n}(R)$ and $U_{n}(R)$  are not strongly $\bar{\sigma}$-IFP and so are not strongly $\bar{\sigma}$-symmetric for $n\geq 2$. Similarly, $D_{n}(R)$ is not strongly $\bar{\sigma}$-symmetric for $n\geq 4$. In \cite[Proposition 3.7]{BHKKL}, it is proved that $D_2(R)$ and $D_3(R)$ are strongly $\bar{\sigma}$-IFP whenever $R$ is a $\sigma$-rigid ring. So, it is natural to ask whether $D_2(R)$ and $D_3(R)$ are strongly $\bar{\sigma}$-symmetric whenever $R$ is a $\sigma$-rigid ring.
In the following example, we eliminate the case for $D_3(R)$.
\begin{example}\rm Note that $D_{3}(\mathbb{R})[x;\bar{\sigma}]\cong D_{3}(\mathbb{R}[x;\sigma])$.
Consider $p=\begin{pmatrix}
1 & x^3 & 0 \\
0&  1 & 0 \\
0 & 0 & 1
\end{pmatrix}$, $q=\begin{pmatrix}
0 & 3x^4 & x^5 \\
0&  0 & x \\
0 & 0 & 0
\end{pmatrix}$ and $r=\begin{pmatrix}
0 & 2x^2 & 3x \\
0&  0 & 0 \\
0 & 0 & 0
\end{pmatrix}$ in $D_{3}(\mathbb{R})[x;\bar{\sigma}]$. Then we have $pqr=0$, but $prq \neq 0$. Thus $D_{3}(\mathbb{R})$ is not strongly $\bar{\sigma}$-symmetric.
\end{example}

In the next theorem, we prove that there exists a subring of $n\times n$ upper triangular matrix ring over a strongly $\sigma$-symmetric ring which is strongly $\bar{\sigma}$-symmetric.
\begin{theorem}\label{mrxextension}
Let $R$ be a ring and $\sigma$ be an endomorphism of $R$. If $R$ is $\sigma$-rigid, then $V_{n}(R)$ is strongly $\bar{\sigma}$-symmetric.
\end{theorem}
\begin{proof}
Firstly, note that $V_{n}(R)[x;\bar{\sigma}]\cong V_{n}(R[x;\sigma])$ and so   every $p\in V_{n}(R)[x;\bar{\sigma}]$ can be expressed in the form $p=(p_1,p_2,\ldots,p_n)$ for some $p_i \in R[x;\sigma]$.
Suppose that $p(x)q(x)r(x)=0$, where
$p(x)=\sum\limits_{i=0}^{m}A_ix^i = (p_{1},p_{2}, \ldots, p_{n}),
  ~q(x)=\sum\limits_{j=0}^{n}B_jx^j = (q_{1},q_{2},\ldots, q_{n})$ and $r(x)=\sum\limits_{k=0}^{l}C_kx^k = (r_{1},r_{2},\ldots, r_{n})\in V_n(R)[x;\bar\alpha]$,
for $A_i=(a_{st}^{(i)})$, $B_j=(b_{uv}^{(j)})$, $C_{k}=(c_{zw}^{(k)}) \in V_n(R)$ for $1\leq u,s,t,u,v,z,w \leq n$.
Then we have the following equalities in $R[x;\sigma]$:
\begin{align}
p_1q_1r_1=&0 \label{mrxss1} \\
p_1q_1r_2+p_1q_2r_1+p_2q_1r_1=&0 \label{mrxss2}\\
p_1q_1r_3+p_1q_2r_2+p_1q_3r_1+p_2q_1r_2+p_2q_2r_1+p_3q_1r_1=&0 \label{mrxss3}\\
\vdots & \notag \\
p_1q_1r_{n-1}+p_1q_{2}r_{n-2}+\ldots+p_1q_{n-1}r_{1}+\ldots+p_{n-1}q_{1}r_1=&0 \label{mrxss4} \\
p_1q_{1}r_{n}+p_1q_2r_{n-1}+\ldots+p_1q_{n}r_1+\ldots+p_{n}q_{1}r_1=&0 \label{mrxss5}
\end{align}
Since $R$ is $\sigma$-rigid, $R[x;\sigma]$ is reduced and so $p(x)q(x)=0$ implies $p(x)R[x;\sigma]q(x)=0$ and $q(x)p(x)=0$. Also $p(x)q(x)^2=0$ implies $p(x)q(x)=0$ for any $p(x),q(x)\in R[x;\sigma]$. We will use these facts in the following procedure without any reference.
By (\ref{mrxss1}), we have $p_1r_1q_1=0$ and $r_1p_1q_1=0$.
If we multiply (\ref{mrxss2}) on the right hand side by $q_1r_1$, then we get $p_2q_1r_1=0$ and so $p_2r_1q_1=0$. Now, (\ref{mrxss2}) becomes
\begin{equation}
\begin{array}[c]{l}
p_1q_1r_2+p_1q_2r_1=0.
\end{array}
\label{mrxss6}
\end{equation}
If we multiply (\ref{mrxss6}) on the right hand side by $q_2r_1$, we have $p_1q_2r_1=0$ and thus $p_1q_1r_2=0$. Hence, we obtain $p_1r_1q_2=0$ and $p_1r_2q_1=0$.
Similarly, if we multiply (\ref{mrxss3}) on the right hand side by $q_1r_1$, then we get $p_3q_1r_1=0$ and $p_3r_1q_1=0$. So, (\ref{mrxss3}) becomes
\begin{equation}
\begin{array}[c]{l}
p_1q_1r_3+p_1q_2r_2+p_1q_3r_1+p_2q_1r_2+p_2q_2r_1=0.
\end{array}
\label{mrxss7}
\end{equation}
If we multiply (\ref{mrxss7}) on the right hand side by $q_2r_1$, then we get $p_2q_2r_1=0$ and $p_2r_1q_2=0$.
Hence, (\ref{mrxss7}) becomes
\begin{equation}
\begin{array}[c]{l}
p_1q_1r_3+p_1q_2r_2+p_1q_3r_1+p_2q_1r_2=0.
\end{array}
\label{mrxss8}
\end{equation}
If we multiply (\ref{mrxss8}) on the right hand side by $q_1r_2$, $q_3r_1$, $q_2r_2$ respectively, we obtain $p_2q_1r_2=0$, $p_1q_3r_1=0$, $p_1q_2r_2=0$ and $p_1q_1r_3=0$.
Inductively, assume that $p_iq_jr_k=0$, where $1\leq i, j, k \leq n-1$ and $i+j+k=n+1$ for $n\geq 2$. If we multiply (\ref{mrxss5}) on the right hand side by $q_1r_1$, we get $p_nq_1r_1=0$ and $p_nr_1q_1=0$. Then (\ref{mrxss5}) becomes
\begin{equation}
\begin{array}[c]{l}
p_1q_{1}r_{n}+p_1q_2r_{n-1}+\ldots+p_1q_{n}r_1+\ldots +p_{n-1}q_2r_1=0.
\end{array}
\label{mrxss9}
\end{equation}
If we multiply (\ref{mrxss9}) on the right hand side by $q_2r_1$, we have $p_{n-1}q_2r_1=0$ and $p_{n-1}r_1q_2=0$.
Continuing this procedure, by multiplying the equation by the appropriate $q_jr_k$ on the right hand side, we get $p_iq_jr_k=0$ and hence, $p_ir_kq_j=0$, where $1\leq i, j, k \leq n$ such that $i+j+k=n+2$.
Consequently, we get $p(x)r(x)q(x)=0$ and therefore, $V_{n}(R)$ is strongly $\bar\sigma$-symmetric.
\end{proof}

By Theorem~\ref{mrxextension}, one may conjecture that if a ring $R$ is strongly $\sigma$-symmetric, then $V_{n}(R)$ is strongly $\bar\sigma$-symmetric. But the following example eliminates this possibility.

\begin{example} \rm \label{trrnotsym}
Consider the ring
$$R=\left\{ \left(\begin{array}{cc} a & \bar{b} \\ 0 & a \\ \end{array} \right)~\mid~a\in \mathbb{Z},\ \bar{b}\in \mathbb{Z}_{4}\right\}$$
and let $\sigma$ be an endomorphism defined by
$\sigma\left(\left(\begin{array}{cc} a & \bar{b}\\ 0 & a \\\end{array}\right)\right)=\left(\begin{array}{cc} a & -\bar{b}\\0 & a\\\end{array}\right)$. By Example~\ref{nonidentity}, we know that $R$ is strongly $\sigma$-symmetric. For
$$A=\left( \begin{array}{cc}
    \left(\begin{array}{cc} 1 & \bar0 \\ 0 & 1 \\ \end{array} \right)
   &\left(\begin{array}{cc} 0 & \bar0 \\ 0 & 0 \\ \end{array} \right)\\
    \left(\begin{array}{cc} 0 & 0 \\ 0 & 0 \\ \end{array} \right)
   &\left(\begin{array}{cc} 1 & \bar0 \\ 0 & 1 \\ \end{array} \right)
    \end{array}\right),
  B=\left( \begin{array}{cc}
    \left(\begin{array}{cc} 0 & \bar1 \\ 0 & 0 \\ \end{array} \right)
   &\left(\begin{array}{cc} -1 & \bar1 \\ 0 & -1 \\ \end{array} \right)\\
    \left(\begin{array}{cc} 0 & 0 \\ 0 & 0 \\ \end{array} \right)
   &\left(\begin{array}{cc} 0 & \bar1 \\ 0 & 0 \\ \end{array} \right)
    \end{array}\right),
  C=\left( \begin{array}{cc}
    \left(\begin{array}{cc} 0 & \bar1 \\ 0 & 0 \\ \end{array} \right)
   &\left(\begin{array}{cc} 1 & \bar1 \\ 0 & 1 \\ \end{array} \right)\\
    \left(\begin{array}{cc} 0 & 0 \\ 0 & 0 \\ \end{array} \right)
   &\left(\begin{array}{cc} 0 & \bar1 \\ 0 & 0 \\ \end{array} \right)
    \end{array}\right)x$$
in $V_2(R)[x;\bar\sigma]$ we have $ABC=0$, but
$$ACB=\left( \begin{array}{cc}
    \left(\begin{array}{cc} 0 & \bar0 \\ 0 & 0 \\ \end{array} \right)
   &\left(\begin{array}{cc} 0 & \bar2 \\ 0 & 0 \\ \end{array} \right)\\
    \left(\begin{array}{cc} 0 & 0 \\ 0 & 0 \\ \end{array} \right)
   &\left(\begin{array}{cc} 0 & \bar0 \\ 0 & 0 \\ \end{array} \right)
    \end{array}\right)\neq 0.$$
Thus $V_{2}(R)$ is not strongly $\bar\sigma$-symmetric.
\end{example}

\begin{corollary}
Let $R$ be a ring and $\sigma$ be an endomorphism of $R$. If $R$ is $\sigma$-rigid, then $R[x]/(x^n)$ is strongly $\bar{\sigma}$-symmetric, where $(x^n)$ denotes the ideal of $R[x]$ generated by $x^n$.
\end{corollary}
\begin{proof}
It is clear since $R[x]/(x^n)\cong V_n(R)$.
\end{proof}

For a given ring $R$ and an $(R,R)$-bimodule $M$, {\em the trivial extension} of $R$ by $M$ is the ring $T(R,M)=R\oplus M$ with the usual addition and the multiplication
\begin{equation*}
(r_1,m_1)(r_2,m_2)=(r_1r_2, r_1m_2+m_1r_2).
\end{equation*}
This is isomorphic to the ring of all matrices
$\begin{pmatrix} r & m\\ 0 & r \end{pmatrix}$, where $r\in R$ and $m\in M$ with the usual matrix operations. Let $\sigma$ be an endomorphism of $R$. We can extend $\sigma$ to an endomorphism $\bar{\sigma}: T(R,R)\rightarrow T(R,R)$ defined by $\bar\sigma\bigg(\begin{pmatrix}
a & b\\ 0& a \end{pmatrix}\bigg)=\bigg(\begin{pmatrix}
\sigma(a) & \sigma(b)\\ 0&\sigma(a) \end{pmatrix}\bigg)$ for $a,b\in R$. Also note that $T(R,0)\cong R$.

\begin{corollary} \label{trr}
Let $R$ be a ring and $\sigma$ be an endomorphism of $R$. If $R$ is $\sigma$-rigid, then the trivial extension $T(R,R)$ of $R$ is strongly $\bar\sigma$- symmetric.
\end{corollary}
\begin{proof}
It is clear by Theorem~\ref{mrxextension}, since $T(R,R)\cong V_2(R).$
\end{proof}

Note that the converse of Corollary~\ref{trr} is not true by Example~\ref{skewarm}(2). By Theorem \ref{mrxextension}, we obtain a generalization of the following results.

\begin{corollary}\cite[Theorem 2.3]{HKKL} If $R$ is a reduced ring, then $R[x]/(x^n)$ is symmetric, where $(x^n)$ denotes the ideal of $R[x]$ generated by $x^n$.
\end{corollary}

\begin{corollary}\cite[Corollary 2.4]{HKKL} Let $R$ be a reduced ring, then $T(R,R)$ is symmetric.

\end{corollary}

\noindent Fatma Kaynarca \\
Department of Mathematics, Afyon Kocatepe University,\\ Afyonkarahisar 03200,
Turkey \\ e-mail: {\tt fkaynarca@aku.edu.tr} \\

\noindent Melis Tekin Akcin \\
Department of Mathematics, Hacettepe University, \\ Ankara 06800,
Turkey \\ e-mail: {\tt hmtekin@hacettepe.edu.tr}

\end{document}